\documentclass[11pt]{amsart}
\usepackage{fullpage,cite,enumitem}
\usepackage{amssymb,amsmath,amsthm,mathtools,extpfeil}
\usepackage[all,cmtip]{xy}
\usepackage{rotating,graphicx}

\usepackage{adjustbox}

\usepackage{tikz}
\usetikzlibrary{cd}
\usetikzlibrary{decorations.markings}
\usetikzlibrary{positioning}

\usepackage[bookmarks]{hyperref}
\hypersetup{colorlinks=true,linkcolor=blue,citecolor=blue}

\newtheorem{theorem}{Theorem}[section]
\newtheorem*{theorem*}{Theorem}
\newtheorem{theoremA}{Theorem}

\newtheorem{lemma}[theorem]{Lemma}

\newtheorem{corollary}[theorem]{Corollary}

\theoremstyle{definition}
\newtheorem{definition}[theorem]{Definition}
\newtheorem{remark}[theorem]{Remark}

\newtheorem{example}[theorem]{Example}

\numberwithin{equation}{section}

\def\G{\mathcal{G}}

\def\R{\mathbb{R}}

\def\sign{\operatorname{sign}}

\def\+{\oplus}

\def\vol{\operatorname{vol}}

\newcommand{\PL}{\mathrm{PL}}

\newcommand{\Dual}{\mathrm{Dual}}

\newcommand{\Lk}{\mathrm{Lk}}

\newcommand{\Asph}{\mathrm{Asph}}

\begin{document}

\title{Quantitative asphericalization and Gromov's linear bordism problem}

\subjclass[2020]{%
  53C23, %	Global geometric and topological methods (à la Gromov); differential geometric analysis on metric spaces
  57Q20. % Cobordism in PL-topology
  %55U10, % Simplicial sets and complexes in algebraic topology
  %18B40. %Groupoids, semigroupoids, semigroups, groups (viewed as categories) [See also 20Axx, 20L05, 20Mxx] 
}

%author one information
\author{Geunho Lim}
\address{Kangnam University, 40 Gangnam-ro, Giheung-gu, Yongin-si, Gyeonggi-do, 16979, Republic of Korea}
\curraddr{}
\email{limg@kangnam.ac.kr}

\begin{abstract}
We prove a PL bordism version of Gromov's linearity conjecture over a broad class of asphericalization groups for manifolds endowed with faithful representations in all dimensions.
Since every group embeds into an asphericalization group, this bordism version of the conjecture holds after allowing the target group to be enlarged.
Our results apply in both the oriented and unoriented settings. 
In the PL category, no bounded local geometry is required. 
As an application, we obtain improved linear bounds for the Cheeger--Gromov $L^2$ $\rho$-invariants of PL $(4k-1)$-manifolds associated with their universal covers, replacing the previous superfactorial growth in dimension by factorial growth. 
The proofs are based on a quantitative asphericalization method that provides explicit control of the complexity of the resulting bordisms.
\end{abstract}
\maketitle

\setcounter{tocdepth}{1}
%\tableofcontents

%%%%%%%%%%%%%%%%%%%%%%%%%%%%%%%%
%%%%%%%%%%%INTRO%%%%%%%%%%%%%%%%
%%%%%%%%%%%%%%%%%%%%%%%%%%%%%%%%

\section{Introduction and main results}
\label{sec:introduction}

Quantitative topology asks not only whether a topological construction exists, but also how complicated such a construction must be.
A basic instance of this principle is Gromov's linearity conjecture~\cite{Gromov:1999-1}.
For a smooth $n$-manifold $M$, Gromov defines its complexity
\[
V(M)
:=
\inf
\left\{
\vol(M,g)
\;\middle|\;
\operatorname{inj}(M,g)\geq 1,\;
|\sec_g|\leq 1
\right\},
\]
with the usual product condition near the boundary when $\partial M\neq\emptyset$.
Gromov conjectured that, whenever a closed smooth $n$-manifold $M$ is null-cobordant, there exists a null-cobordism $W$ satisfying $V(W)\leq C(n) \cdot V(M)$, where $C(n)$ depends only on the dimension.

There is a natural PL analogue.
For a PL $n$-manifold $M$, let $\Delta(M)$ denote the minimum number of $n$-simplices among all PL triangulations of $M$.
One then asks whether every null-cobordant PL manifold admits a null-cobordism $W$ with $\Delta(W)\leq C(n) \cdot \Delta(M)$.
The linear problem remains open in this generality.
In the smooth category, Chambers, Dotterrer, Manin, and Weinberger~\cite{Chambers-Dotterrer-Manin-Weinberger:2018-1} proved that, for every $\epsilon>0$, one can obtain a bound of order $V(M)^{1+\epsilon}$.
Manin and Weinberger~\cite{Manin-Weinberger:2023-1} obtained an analogous almost-linear statement in the PL category under bounded local geometry.

A bordism version of Gromov's problem was introduced by Cha and the author~\cite{Cha-Lim:2024-1}.
A manifold over a discrete group $G$ is a manifold equipped with a map to $BG$; after choosing basepoints, we denote the associated representation by $\varphi\colon\pi_1(M)\longrightarrow G$.
Throughout this paper, by a \emph{representation} of the fundamental group in a discrete group $G$ we mean a group homomorphism $\varphi\colon\pi_1(M)\longrightarrow G$, and we call the representation \emph{faithful} if $\varphi$ is injective.
For a map $f\colon X\longrightarrow Y$, we say that $f$ is \emph{$\pi_1$-injective} if the induced representation $f_*\colon\pi_1(X)\longrightarrow\pi_1(Y)$ is faithful.
A bordism over $G$ from $M$ to $N$ is called a \emph{bordism from $M$ to a trivial end} if the structure map $N\to BG$ is homotopic to a constant map; see~\cite[Definition~3.1]{Cha:2014-1}.
The quantitative problem is to determine when one can find such a bordism $W$ satisfying $\Delta(W)\leq C(n)\cdot\Delta(M)$, with $C(n)$ independent of $M$, of the target group, and of the structure map.

This formulation naturally raises two related but logically distinct questions: over which groups and representations does such a linear bordism theorem hold, and how efficiently can the multiplicative constant $C(n)$ be controlled?
The problem itself is posed for arbitrary discrete target groups, but the presently known general constructions apply only under additional algebraic or geometric hypotheses.
For arbitrary representations, Cha and the author~\cite{Cha-Lim:2024-1} established linearity over a large family of acyclic groups arising from Baumslag--Dyer--Heller-type acyclic constructions, whereas for faithful representations the author and Weinberger~\cite{Lim-Weinberger:2023-1} obtained linear bordisms using groups arising from relative hyperbolization.
These constitute the two general mechanisms previously available for producing such linear bordisms.
Thus there are two complementary issues: the scope of the linear bordism theorem with respect to the target group and representation, and the dimension dependence of the multiplicative constant $C(n)$.
The present paper addresses both from the viewpoint of faithful representations, enlarging the class of available target groups through asphericalization while substantially reducing the quantitative cost of the geometric construction.

These two previously known mechanisms have quite different scopes and quantitative costs.
The acyclic-group method of Cha and the author~\cite{Cha-Lim:2024-1} applies to arbitrary representations, but its proof passes through quantitative chain null-homotopies, simplicial transversality, and an iteration over the skeleta of the classifying space.
The hyperbolization method of the author and Weinberger~\cite{Lim-Weinberger:2023-1} requires the representation to be faithful, but replaces this algebraic machinery by a geometric construction and comes with a substantially smaller explicit construction bound.
The present asphericalization method also assumes faithfulness, but further removes the subdivision and hyperbolization costs.
We give a quantitative comparison of the constants produced by these three methods after stating the main theorem.

The starting point of the present paper is that the bordism argument does not require the full nonpositive-curvature structure supplied by hyperbolization.
What is essential is a more topological package: preservation of the manifold structure together with sufficiently strong control of fundamental groups, particularly the $\pi_1$-injectivity of the inclusions of distinguished subspaces.
This leads us to replace hyperbolization by a quantitative \emph{asphericalization}.

Asphericalization goes back to Kan and Thurston~\cite{Kan-Thurston:1976-1}, who constructed functorial aspherical models retaining the homological information of a given space.
Maunder~\cite{Maunder:1981} subsequently developed an inductive simplex-by-simplex version of this idea.
A geometric counterpart is Gromov's hyperbolization~\cite{Gromov:1987-1}, developed further by Davis and Januszkiewicz~\cite{Davis-Januszkiewicz:1991-1} and by Charney and Davis~\cite{Charney-Davis:1995}.
In fact, Davis and Januszkiewicz viewed hyperbolization as a geometrically controlled form of Kan--Thurston asphericalization, adapted to preserve manifold and link structures.
Relative forms of hyperbolization have proved particularly useful in bordism problems~\cite{Davis-Januszkiewicz-Weinberger:2001,Lim-Weinberger:2023-1}.

The construction developed here returns to this underlying asphericalization principle while retaining precisely the additional structure needed for quantitative bordism.
We introduce coherent systems of asphericalized simplices and a direct simplex-replacement procedure which preserves PL manifold structures and provides strong subcomplex $\pi_1$-injectivity.
We then construct an explicit system of such simplices, called \emph{toroidal simplices}, by an inductive reflection-cylinder construction.
The associated groups will be called \emph{asphericalization groups}; see Definition~\ref{def:asph-gr}.

The new point is not the reflection-cylinder construction by itself, but its organization into a strictly coherent simplex system with aspherical and $\pi_1$-injective face unions.
This coherence makes it possible to replace the simplices of an arbitrary ordered complex directly, without passing to barycentric subdivision, while preserving the manifold structure and retaining explicit quantitative control.
The main result is the following.

\begin{theoremA}\label{thm:quantitative-bordism}
Let $G$ be a discrete group, and let $M$ be a connected closed PL $n$-manifold equipped with a faithful representation $\varphi\colon \pi_1(M)\longrightarrow G$.
Then there exists an asphericalization group $\Gamma$ containing $G$ and a PL bordism $W$ over $\Gamma$ from $M$ to a trivial end such that
\[
\Delta(W)\leq C(n) \cdot \Delta(M),
\]
where $C(n) = (n+3) \cdot 4^n \cdot (n+1)! \cdot (n+2)!$.

If $M$ is oriented, then $W$ is oriented.
\end{theoremA}

In particular, every group embeds into an asphericalization group of the type appearing in Theorem~\ref{thm:quantitative-bordism}.
Thus the bordism version of Gromov's linearity problem has an affirmative answer for faithful representations after allowing the target group to be enlarged.
The point of the theorem is not merely the existence of such an enlargement: the asphericalization procedure gives a direct quantitative construction whose dimension dependence is of factorial type.

\subsubsection*{Comparison of the quantitative constants}

The improvement in the multiplicative constant is one of the main quantitative features of Theorem~\ref{thm:quantitative-bordism}.
To compare the available constructions on the same scale, we measure complexity throughout by the number of top-dimensional simplices and consider the explicit constants furnished by the respective constructions.
The detailed estimates are given in Appendix~\ref{app:quantitative-comparison}; here we record only their dimension dependence.

For the acyclic-group construction of Cha and the author~\cite{Cha-Lim:2024-1}, which applies to arbitrary representations, tracing the skeleton-by-skeleton bordism construction together with the improved quantitative Baumslag--Dyer--Heller chain homotopy of~\cite{Lim:2022-1} gives a superexponential bound
\[
\log C_{\mathrm{acy}}(n)
=
O\!\bigl(n^3\log(n+2)\bigr).
\]
For faithful representations, the relative-hyperbolization construction of the author and Weinberger~\cite{Lim-Weinberger:2023-1} gives 
\[
\log C_{\mathrm{hyp}}(n) = n^2\log n+O(n^2).
\]
By contrast, the direct asphericalization constructed in the present paper gives 
\[
\log C_{\mathrm{asph}}(n) = 2n\log n+O(n).
\]

Thus, in the common setting of faithful representations, the dimension dependence drops and
\[
\frac{\log C_{\mathrm{asph}}(n)}
{\log C_{\mathrm{hyp}}(n)}
\longrightarrow
0.
\]
The direct asphericalization replaces the superfactorial cost of relative hyperbolization by factorial-type growth.
The acyclic-group construction has the greater generality of allowing arbitrary representations, whereas the two geometric constructions available under faithfulness yield the substantially smaller explicit construction bounds described above.

\subsubsection*{Method of the proof}
For a relative asphericalization
\[
\mathcal X(K,J),
\qquad
J=J_1\sqcup\cdots\sqcup J_r\subseteq\partial K,
\]
the asphericalized-simplex condition implies a strong peripheral $\pi_1$-injectivity property.
Each boundary component $J_i$ remains $\pi_1$-injective after any collection of the other boundary components is coned off.
This is proved from the subcomplex $\pi_1$-injectivity built into the simplex system, together with the local cone structure and a tree-of-spaces argument.

We apply this construction to $\bigl(M\times I,M_-\sqcup M_+\bigr)$.
After coning off $M_-$, the inclusion of $M_+$ still induces an injection on fundamental groups.
Amalgamating the resulting group with the original target group $G$ along the faithful image of $\pi_1(M)$ produces the asphericalization group $\Gamma$ over which the required bordism is defined.

The quantitative improvement comes from the direct nature of the simplex-replacement construction.
Relative hyperbolization requires barycentric subdivision followed by replacement by hyperbolized simplices whose top-dimensional complexity is superfactorial in the dimension.
The present construction instead replaces each simplex directly by a toroidal simplex of factorial-type complexity.
Quantifying the triangulation of the relative construction, together with the explicit complexity estimate for toroidal simplices, gives $\Delta(W) \leq C(n) \cdot \Delta(M)$, where $C(n)=O\!\left(n\, \cdot 4^n\, \cdot (n+1)!\, \cdot (n+2)! \right)$.
Thus the construction retains the asphericity and $\pi_1$-injectivity needed for the bordism argument while dispensing with the nonpositive-curvature machinery, and with a substantial part of the combinatorial cost, of hyperbolization.

A principal application concerns the Cheeger--Gromov $L^2$ $\rho$-invariant.
For a closed oriented $(4k-1)$-manifold $M$, Cheeger and Gromov~\cite{Cheeger-Gromov:1985-1,Cheeger-Gromov:1985-2} introduced an invariant associated with a representation $\pi_1(M)\longrightarrow G$.
Analytically, it is defined as a difference of $\eta$-invariants.
The topological interpretation of Chang and Weinberger~\cite{Chang-Weinberger:2003-1} expresses it as an $L^2$-signature defect of a suitable bounding manifold.
Throughout this paper, we adopt their topological definition for manifolds endowed with a faithful representation $\varphi\colon \pi_{1}(M) \to G$.

\begin{definition}\label{def:CWrho}
For a closed oriented $(4k-1)$-manifold $M$ endowed with a faithful representation $\varphi\colon\pi_1(M)\longrightarrow G$, the Cheeger--Gromov $L^2$ $\rho$-invariant is defined as the $L^2$-signature defect
\[
\rho^{(2)}(M):=\frac{1}{r}\bigl(\sign_{\Gamma}^{(2)}(W)-\sign(W)\bigr)\in\R,
\]
where $W$ is a compact oriented $4k$-manifold with $\partial W=rM$ such that each boundary inclusion is $\pi_1$-injective, and where there are a group $\Gamma$, a monomorphism $i_G\colon G\longrightarrow\Gamma$, and a representation $\psi\colon\pi_1(W)\longrightarrow\Gamma$ extending $i_G\circ\varphi$ on the boundary.
Here $\sign_{\Gamma}^{(2)}(W)$ denotes the $L^2$-signature associated with the representation $\psi$.
\end{definition}

The existence of such a manifold $W$ follows from Thom's classical work on cobordism and Hausmann~\cite{Hausmann:1981}.  
Such a group $\Gamma$, together with $i_G$ and $\psi$, exists by the standard amalgamation construction.
The well-definedness of the $\rho$-invariant is a consequence of the $\Gamma$-induction property for $L^2$-signatures~\cite{Cheeger-Gromov:1985-1} together with the standard Novikov additivity argument~\cite{Chang-Weinberger:2003-1}.
For the general definition associated with an arbitrary representation, we refer to~\cite{Cha:2014-1}.

Cheeger and Gromov proved a linear volume estimate under bounded local geometry: if $M$ is a closed Riemannian $d$-manifold with $\operatorname{inj}(M)\geq\iota$ and $|\sec|\leq K_0$, then
\[
\bigl|\rho^{(2)}(M)\bigr|
\leq
C(d,\iota,K_0) \cdot \vol(M).
\]
Although the invariant itself is topological, the original proof is analytic and depends on curvature and injectivity-radius bounds.
This led to the problem of obtaining such estimates by purely topological means.

Cha~\cite{Cha:2014-1} initiated a quantitative topological approach using the $L^2$-signature defect and obtained the first explicit linear simplicial bounds.
Further improvements were obtained in~\cite{Lim:2022-1,Cha-Lim:2024-1,Lim-Weinberger:2023-1}.
Combining Theorem~\ref{thm:quantitative-bordism} with this signature-defect method gives the following bound for the invariant associated with the universal cover because a faithful representation has trivial kernel.

\begin{theoremA}\label{thm:efficient-rho}
Let $n=4k-1$ with $k\geq 1$, and let $M$ be a closed oriented PL $n$-manifold endowed with a faithful representation $\varphi\colon \pi_1(M) \to G$.
Then
\[
\left|
\rho^{(2)}(M)
\right|
\leq
C(n) \cdot \Delta(M),
\]
where $C(n) = O\!\Bigl(\tfrac{1}{\sqrt{n+1}}\;\cdot n \cdot 8^{\,n}\,\cdot (n+1)!\,\cdot(n+2)!\Bigr)$.
\end{theoremA}

The essential quantitative gain can be seen by comparison with the hyperbolization bound of~\cite{Lim-Weinberger:2023-1}.
There the corresponding constant has order
\[
O\!\left(
\frac{1}{\sqrt{n+1}}\,
\cdot
n^2
\cdot
8^n
\cdot
\Bigl(\prod_{k=1}^{n} k!\Bigr)
\cdot
\Bigl(\prod_{k=1}^{n+1} k!\Bigr)
\right).
\]
Hence the two superfactorial products occurring in the previous estimate are replaced by the ordinary factorial factors.
This is the quantitative reflection of the main structural difference between the two constructions: the present asphericalization avoids the subdivision cost inherent in the hyperbolization procedure.

\subsection*{Organization of the paper}

Section~\ref{sec:asph} develops coherent PL simplex systems, direct simplex replacement, and the structural asphericity and $\pi_1$-injectivity results needed later.
Section~\ref{sec:toroidal-simplex} constructs the toroidal asphericalized simplex system by reflection cylinders.
Section~\ref{sec:rel-asph} develops the quantitative relative asphericalization, proves peripheral $\pi_1$-injectivity, and proves Theorem~\ref{thm:quantitative-bordism}.
Section~\ref{sec:linear-rho} applies the bordism theorem to the Cheeger--Gromov $L^2$ $\rho$-invariant and proves Theorem~\ref{thm:efficient-rho}.
Finally, Section~\ref{sec:smooth} discusses the smooth analogue and the quantitative dependence of the corresponding constants, while Appendix~\ref{app:quantitative-comparison} gives a detailed comparison of the multiplicative constants arising from the acyclic-group, relative-hyperbolization, and direct-asphericalization constructions.

\subsection*{Acknowledgements}

The author is grateful to Shmuel Weinberger, Jae Choon Cha, and Fedor Manin for many stimulating conversations related to this work.
This work was supported by Kangnam University Research Grants (2026).
The author used ChatGPT (GPT-5) to assist with typesetting, language editing of portions of the text, figure creation, and adversarial reading to identify potential gaps or inconsistencies in the exposition and arguments.
The author takes full intellectual responsibility for all content of this paper, including its mathematical results, proofs, figures, and references.

%%%%%%%%%%%%%%%%%%%%%%%%%%%%%%
%%%%%%%Section 2%%%%%%%%%%%%%%
%%%%%%%%%%%%%%%%%%%%%%%%%%%%%%

\section{Asphericalized simplices and asphericalization}
\label{sec:asph}

In this section, we introduce an asphericalization procedure.
The construction takes as input a coherent family of PL manifolds with faces that serve as replacements for the standard simplices.
Each such manifold has the same labeled face poset and local normal-incidence structure as the corresponding standard simplex, while its interior is allowed to have nontrivial topology.

Throughout this section, all simplicial complexes are finite, and all manifolds and maps are PL.
For notational convenience, we shall not distinguish between a simplicial complex and its geometric realization when no confusion is likely to arise.

\subsection{Coherent PL simplex systems}

We first fix notation and recall the definition of a PL manifold.

\begin{definition}
\label{def:PLmfd}
A finite $m$-dimensional simplicial complex $K$ is called a \emph{PL $m$-manifold with possibly empty boundary} if, for every $\ell$-simplex $a\in K$,
\[
\Lk_K(a)\cong_{\PL}
\begin{cases}
S^{m-\ell-1}, & |a|^{\circ}\subset \operatorname{int} K,\\[2mm]
B^{m-\ell-1}, & |a|^{\circ}\subset \partial K.
\end{cases}
\]
We use the standard conventions in dimension $-1$.
\end{definition}

For $k\geq 0$, let $[k]:=\{0,1,\ldots,k\}$, equipped with its natural ordering, and let $\Delta^k$ denote the standard $k$-simplex with ordered vertex set $[k]$.
For a nonempty subset $S=\{s_0<\cdots<s_j\}\subseteq[k]$, let $\iota_S\colon[j]\hookrightarrow[k]$ be the order-preserving injection defined by $\iota_S(r)=s_r$.
We denote the corresponding face of $\Delta^k$ by $\Delta^k[S]:=\operatorname{conv}\{e_s\mid s\in S\}=\iota_S(\Delta^j)$.

Let $\emptyset\neq T\subseteq S\subseteq[k]$, where $S=\{s_0<\cdots<s_j\}$.
Define $\operatorname{pos}_S(T):=\{\,r\in[j]\mid s_r\in T\,\}\subseteq[j]$.
With this notation, $\iota_T=\iota_S\circ\iota_{\operatorname{pos}_S(T)}$.

For a finite set $A$, let $[0,\varepsilon)^A$ denote the orthant whose coordinates are indexed by $A$.
Fix a compatible family of standard PL face collars $\kappa_S^k\colon\Delta^{|S|-1}\times[0,\varepsilon)^{[k]\setminus S}\longrightarrow\Delta^k$ such that $\kappa_S^k(y,0)=\iota_S(y)$.
We assume that these collars are compatible under nested face inclusions.

A \emph{PL $m$-manifold with faces} is a compact PL $m$-manifold $X$ with possibly empty boundary, together with a finite collection of closed codimension-one PL submanifolds $\mathcal F(X)=\{F_1,\ldots,F_N\}$, called its \emph{facets}, satisfying the following conditions.
First, $\partial X=\bigcup_{\alpha=1}^N F_\alpha$.
Second, if $x\in X$ lies in precisely the distinct facets $F_{\alpha_1},\ldots,F_{\alpha_r}$, then there are a neighborhood $U$ of $x$, a number $\varepsilon>0$, and a PL homeomorphism $\phi\colon U \xrightarrow{\cong_{\PL}} (-\varepsilon,\varepsilon)^{m-r}\times[0,\varepsilon)^r$ sending $x$ to the origin and satisfying $\phi(U\cap F_{\alpha_i}) = \phi(U)\cap\{t_i=0\}, 1\leq i\leq r$, where $t_1,\ldots,t_r$ are the coordinates of the orthant factor.
A \emph{face} of $X$ is a connected component of a nonempty intersection of facets, with $X$ itself regarded as the unique codimension-zero face.
Thus every codimension-$r$ face is locally given by the simultaneous vanishing of $r$ distinct normal coordinates.
All faces are understood to be closed.

\begin{definition}[Coherent PL simplex system]
\label{def:coherent-simplex-system}
A \emph{coherent PL simplex system of length $n$} consists of the following data for each $0\leq k\leq n$.

\begin{enumerate}
\item[\textnormal{(S1)}]
A compact connected PL $k$-manifold $X_k$ with faces.

\item[\textnormal{(S2)}]
For every nonempty subset $S\subseteq[k]$, a PL embedding $j_S^k\colon X_{|S|-1}\hookrightarrow X_k$.
The \emph{$S$-face} of $X_k$ is defined by $X_k[S]:=j_S^k\bigl(X_{|S|-1}\bigr)\subseteq X_k$.
We also set $X_k[\emptyset]:=\emptyset$.

\item[\textnormal{(S3)}]
The face embeddings are strictly coherent.
More precisely, $j_{[k]}^k=\operatorname{id}_{X_k}$, and for every $\emptyset\neq T\subseteq S\subseteq[k]$, one has $j_T^k=j_S^k\circ j_{\operatorname{pos}_S(T)}^{\,|S|-1}$.
Equivalently, the following diagram commutes:
\[
\begin{tikzcd}
X_{|T|-1}
\arrow[r, "{j_{\operatorname{pos}_S(T)}^{\,|S|-1}}"]
\arrow[d, "{j_T^k}"']
&
X_{|S|-1}
\arrow[d, "{j_S^k}"]
\\
X_k
\arrow[r, equal]
&
X_k.
\end{tikzcd}
\]

\item[\textnormal{(S4)}]
The distinguished faces have the same incidence pattern as the faces of $\Delta^k$.
For all $S,T\subseteq[k]$, one has $X_k[S]\cap X_k[T]=X_k[S\cap T]$.
The facets of $X_k$ are precisely the distinguished codimension-one faces $X_k\bigl[[k]\setminus\{i\}\bigr], 0\leq i\leq k$ and $\partial X_k=\bigcup_{i=0}^k X_k\bigl[[k]\setminus\{i\}\bigr]$.

\item[\textnormal{(S5)}]
The distinguished faces admit compatible product collars.
For every nonempty subset $S\subseteq[k]$, there is a PL embedding $c_S^k\colon X_{|S|-1}\times[0,\varepsilon)^{[k]\setminus S}\longrightarrow X_k$ onto a neighborhood of $X_k[S]$ such that $c_S^k(x,0)=j_S^k(x)$.

These collars are compatible under nested face inclusions.
More precisely, let $\emptyset\neq T\subseteq S\subseteq[k]$ and use the canonical decomposition $[k]\setminus T=(S\setminus T)\sqcup([k]\setminus S)$.
After identifying the coordinates indexed by $S\setminus T$ with those indexed by $[|S|-1]\setminus\operatorname{pos}_S(T)$, we require
\[
c_T^k(x,(u,v))
=
c_S^k\left(
c_{\operatorname{pos}_S(T)}^{\,|S|-1}(x,u),v
\right)
\]
whenever both sides are defined.

\item[\textnormal{(S6)}]
There is a PL \emph{structure map} $f_k\colon X_k\longrightarrow\Delta^k$ satisfying the following conditions.

\begin{enumerate}
\item[\textnormal{(a)}]
For every nonempty subset $S\subseteq[k]$, one has $f_k\circ j_S^k=\iota_S\circ f_{|S|-1}$.

\item[\textnormal{(b)}]
The distinguished faces are the exact inverse images of the corresponding faces of $\Delta^k$, that is, $f_k^{-1}\bigl(\Delta^k[S]\bigr)=X_k[S]$.
Equivalently, $f_k^{-1}\bigl((\Delta^k[S])^{\circ}\bigr)=X_k[S]^{\circ}$, where $X_k[S]^{\circ}:=X_k[S]\setminus\bigcup_{T\subsetneq S}X_k[T]$.

\item[\textnormal{(c)}]
The structure map has product form in the chosen collars.
More precisely, $f_k\bigl(c_S^k(x,u)\bigr)=\kappa_S^k\bigl(f_{|S|-1}(x),u\bigr)$ for every $x\in X_{|S|-1}$ and $u\in[0,\varepsilon)^{[k]\setminus S}$.
\end{enumerate}
\end{enumerate}
\end{definition}

\begin{remark}
For notational convenience, from now on we work with simplex systems defined in all dimensions, unless a finite length is explicitly specified.
Thus $\mathcal X=\{(X_k,f_k)\}_{k\geq0}$ denotes an infinite coherent PL simplex system, and $\mathcal X_{\leq n}:=\{(X_k,f_k)\}_{k=0}^n$ denotes its truncation to length $n$.
All constructions involving a finite-dimensional simplicial complex use only the corresponding finite truncation of $\mathcal X$.
\end{remark}

\begin{remark}
\label{rem:simplex-face-meaning}
The term ``simplex system'' reflects the fact that the collection $\{X_k[S]\mid\emptyset\neq S\subseteq[k]\}$ has the same labeled face poset as $\Delta^k$.
In particular, every $(j+1)$-element subset $S\subseteq[k]$ determines a distinguished face $X_k[S]\cong_{\PL}X_j$.
Condition~\textnormal{(S3)} ensures that a lower-dimensional face included directly into $X_k$ agrees with the same face obtained through any intermediate face.
Condition~\textnormal{(S4)} excludes unintended intersections between distinguished faces and identifies the distinguished codimension-one faces with the facets of the manifold-with-faces structure.
The manifold-with-faces condition gives the local orthant model for their transverse incidence, while condition~\textnormal{(S5)} strengthens this local structure by choosing product collars that are strictly compatible under all nested face inclusions.
\end{remark}

\begin{definition}[Oriented degree-one simplex system]
\label{def:oriented-degree-one-system}
A coherent PL simplex system is called \emph{oriented} if each $X_k$ is oriented and the orientations are compatible with the standard simplex convention.
More precisely, for every $k\geq1$ and $0\leq i\leq k$, the boundary orientation on the facet $X_k\bigl[[k]\setminus\{i\}\bigr]$ corresponds under $j_{[k]\setminus\{i\}}^k\colon X_{k-1}\longrightarrow X_k\bigl[[k]\setminus\{i\}\bigr]$ to $(-1)^i$ times the fixed orientation of $X_{k-1}$.

An oriented coherent PL simplex system of length $n$ is called \emph{degree one} if, for every $0\leq k\leq n$, the map
$f_k\colon(X_k,\partial X_k)\longrightarrow(\Delta^k,\partial\Delta^k)$
has degree $+1$.
For an infinite coherent simplex system, degree one means that every finite truncation is degree one.
\end{definition}

By condition~\textnormal{(S6)(a)}, the restriction of $f_k$ to the face $X_k[S]$, under the identification $X_k[S]\cong X_{|S|-1}$, agrees with $f_{|S|-1}$.
It follows that the structure maps have degree one on every distinguished face.

Let $P\subseteq\Delta^k$ be a simplicial subcomplex.
Define the corresponding union of distinguished faces by $X_k[P]:=\bigcup_{\Delta^k[S]\subseteq P}X_k[S]$.
For the empty subcomplex, set $X_k[\emptyset]:=\emptyset$.
Condition~\textnormal{(S6)(b)} gives $X_k[P]=f_k^{-1}(|P|)$.

\begin{definition}[Asphericalized simplex system]
\label{def:asphericalized-simplex-system}
An oriented degree-one coherent PL simplex system $\mathcal X_{\leq n}=\{(X_k,f_k)\}_{k=0}^n$ is called an \emph{asphericalized simplex system} if it satisfies the following condition.

\begin{enumerate}
\item[\textnormal{(A)}]
For every $0\leq k\leq n$ and every simplicial subcomplex $P\subseteq\Delta^k$, each path component of $X_k[P]$ is aspherical, and its inclusion into $X_k$ induces a monomorphism on fundamental groups.
\end{enumerate}

The space $X_k$, together with its distinguished face structure and structure map $f_k\colon X_k\to\Delta^k$, is called an \emph{asphericalized $k$-simplex}.
\end{definition}

\begin{remark}
\label{rem:low-dimensional-examples}
The standard simplices $\Delta^0$ and $\Delta^1$, equipped with their usual face structures and identity maps, are asphericalized simplices.
The standard simplex $\Delta^2$ is not an asphericalized $2$-simplex.
In fact, for $P=\partial\Delta^2$, the homomorphism $\pi_1(\partial\Delta^2)\longrightarrow\pi_1(\Delta^2)$ induced by inclusion is not injective.
On the other hand, a punctured torus whose boundary is decomposed into three coherently labeled $1$-faces gives a basic example of an asphericalized $2$-simplex.
However, this example does not extend naively to higher dimensions.
For $n\geq3$, a punctured $n$-torus $T^n\setminus\operatorname{int}D^n$ is not an asphericalized $n$-simplex because, taking $P=\partial\Delta^n$, we obtain $X_n[P]=\partial X_n\cong S^{n-1}$, which is not aspherical.
In Section~\ref{sec:toroidal-simplex}, we construct higher-dimensional asphericalized simplices.
\end{remark}

\subsection{Direct simplex replacement}

Let $K$ be a finite simplicial complex equipped with a linear ordering of its vertex set.
For a simplex $\sigma=[v_0<\cdots<v_k]\in K$, let $a_\sigma\colon\Delta^k\xrightarrow{\cong} \sigma $ be the order-preserving affine identification.

For a face $\tau\subseteq\sigma$, define $S_\sigma(\tau):=\{\,i\in[k]\mid v_i\in\tau\,\}$.
The ordering of $\tau$ induced from $K$ agrees with the ordering of the subset $S_\sigma(\tau)$.

\begin{definition}[Direct simplex replacement and asphericalization]
\label{def:replacement}
Let $\mathcal X_{\leq n}=\{(X_k,f_k)\}_{k=0}^n$ be a coherent PL simplex system, and let $K$ be a finite ordered simplicial complex of dimension at most $n$.

For each $k$-simplex $\sigma\in K$, take a labeled copy $X(\sigma)$ of $X_k$.
For every face inclusion $\tau\subseteq\sigma$, where $\dim\tau=\ell$, identify $X(\tau)$ with the distinguished face $X_k[S_\sigma(\tau)]\subseteq X(\sigma)$ using the specified embedding $j_{S_\sigma(\tau)}^k\colon X_\ell\hookrightarrow X_k$.

The \emph{direct simplex replacement} of $K$ with respect to $\mathcal X_{\leq n}$ is the colimit
\[
\mathcal X(K)
:=
\operatorname*{colim}_{\sigma\in\operatorname{Face}(K)}
X(\sigma).
\]
Equivalently,
\[
\mathcal X(K)
=
\left(
\bigsqcup_{\sigma\in K}X(\sigma)
\right)\Big/\sim,
\]
where the equivalence relation is generated by $x\sim j_{S_\sigma(\tau)}^{\dim\sigma}(x)$ for every face inclusion $\tau\subseteq\sigma$ and every $x\in X(\tau)$.

On each replacement simplex $X(\sigma)\cong X_k$, define $F_K|_{X(\sigma)}:=a_\sigma\circ f_k$.
Condition~\textnormal{(S6)(a)} gives $f_k\circ j_{S_\sigma(\tau)}^k=\iota_{S_\sigma(\tau)}\circ f_\ell$, so these maps agree on every common face.
They therefore induce a well-defined PL map $F_K\colon\mathcal X(K)\longrightarrow K$, which we call the \emph{structure map}.

If $\mathcal X_{\leq n}$ is asphericalized, we refer to $\mathcal X(K)$ as the \emph{asphericalization of $K$ with respect to $\mathcal X_{\leq n}$}.
\end{definition}

\begin{remark}
\label{rem:replacement-gluing}
The construction does not require an auxiliary choice of a homeomorphism between two faces.
If $\sigma\cap\sigma'=\tau$, then the replacement simplices $X(\sigma)$ and $X(\sigma')$ are identified through the common model $X(\tau)$.
Thus, their union is the pushout $X(\sigma)\cup_{X(\tau)}X(\sigma')$.
Strict coherence guarantees that the attaching maps agree on all iterated faces, while the face-intersection condition prevents unintended identifications.
Consequently, $X(\sigma)\cap X(\sigma')=X(\sigma\cap\sigma')$ inside $\mathcal X(K)$.
\end{remark}

\begin{remark}
\label{rem:structure-map-well-defined}
The block maps $F_\sigma$ are compatible with all face identifications.
In fact, let $\tau\subseteq\sigma$, where $\dim\sigma=k$ and $\dim\tau=\ell$, and put $S:=S_\sigma(\tau)$.
On the copy $X(\tau)\cong X_\ell$, condition~\textnormal{(S6)(a)} gives
\[
F_\sigma\circ j_S^k
=
a_\sigma\circ f_k\circ j_S^k
=
a_\sigma\circ\iota_S\circ f_\ell
=
a_\tau\circ f_\ell
=
F_\tau.
\]
Here $a_\sigma\circ\iota_S=a_\tau$ because the ordering of the vertices of $\tau$ induced from $\sigma$ agrees with the global ordering of $V(K)$.
Therefore, the maps $F_\sigma$ descend uniquely to a continuous map $F_K\colon\mathcal X(K)\to K$ by the universal property of the colimit.
Since $K$ is finite and all block maps and attaching maps are PL, the induced map $F_K$ is PL.
\end{remark}

For a subcomplex $L\subseteq K$, denote the corresponding subcolimit by $\mathcal X(L)\subseteq\mathcal X(K)$.
The exact face-preimage condition implies $F_K^{-1}(L)=\mathcal X(L)$.
In particular, $F_K^{-1}(\partial K)=\mathcal X(\partial K)$.

\begin{remark}\label{rem:compare-williams}
The construction above is a subdivision-free analogue of the Williams fiber-product construction used in the Davis--Januszkiewicz hyperbolization.
In the latter construction, the barycentric subdivision of an $n$-dimensional simplicial complex $K$ carries the canonical dimension map $d\colon\operatorname{sd}K\to\Delta^n$, and a space over $\Delta^n$ is pulled back along this map; see~\cite[Section~1]{Davis-Januszkiewicz:1991-1} and~\cite{Williams:1963}.
Thus, barycentric subdivision supplies the global map to $\Delta^n$ that makes the Williams construction functorial.

In the present construction, the global map $\operatorname{sd}K\to\Delta^n$ is replaced by the coherent ordered face embeddings $j_S^k$ and the compatible structure maps $f_k\colon X_k\to\Delta^k$.
Consequently, each simplex $\sigma$ of $K$ is replaced directly by the block $X_{\dim\sigma}$, and the replacement blocks are glued through their prescribed lower-dimensional faces without first subdividing $K$.
The quantitative advantage is that the combinatorial growth caused by barycentric subdivision is avoided; the resulting complexity estimates will be established below.

The construction is natural with respect to order-preserving nondegenerate simplicial maps.
In fact, if $g\colon K\to K'$ is such a map and $\sigma=[v_0<\cdots<v_k]$ is a simplex of $K$, then $g(\sigma)=[g(v_0)<\cdots<g(v_k)]$ is a $k$-simplex of $K'$, and the canonical identification of the two labeled copies of $X_k$ defines a block map $X(\sigma)\to X(g(\sigma))$.
Since $S_\sigma(\tau)=S_{g(\sigma)}(g(\tau))$ for every face $\tau\subseteq\sigma$, these block maps respect all face identifications and induce a PL map $\mathcal X(g)\colon\mathcal X(K)\to\mathcal X(K')$ satisfying $F_{K'}\circ\mathcal X(g)=|g|\circ F_K$.
A simplicial map that reverses the chosen vertex order or collapses a simplex does not, in general, induce a canonical map between the replacement spaces unless the simplex system is equipped with additional equivariance or degeneracy data.

Finally, since the construction does not pass to $\operatorname{sd}K$, it does not automatically inherit the flag structure supplied by barycentric subdivision and used in the usual CAT(0) link criterion.
Accordingly, no nonpositive-curvature conclusion is asserted here.
The construction retains instead the coherent face and local normal-incidence data needed for asphericity, preservation of PL manifold structures, and quantitative control.
\end{remark}

We first record two elementary properties of the replacement construction.
For the remainder of this paper, we assume that $K$ is a connected finite ordered simplicial complex of dimension $n$ and that $\mathcal X$ is an oriented degree-one coherent PL simplex system.

\begin{lemma}[Connectivity and surjectivity on fundamental groups]
\label{lem:pi1-surj}
The direct simplex replacement $\mathcal X(K)$ is path-connected.
Moreover, for every vertex $v\in K$ and the corresponding point $\bar v\in\mathcal X(K)$, the structure map induces a surjection
\[
(F_K)_*\colon\pi_1(\mathcal X(K),\bar v)\twoheadrightarrow\pi_1( K,v).
\]
Consequently, $(F_K)_*\colon H_1(\mathcal X(K);\mathbb Z)\twoheadrightarrow H_1( K;\mathbb Z)$ is surjective.
\end{lemma}

\begin{proof}
By condition~\textnormal{(S1)}, $X_0$ is a compact connected $0$-dimensional PL manifold and hence consists of a single point.
Thus every vertex $v$ of $K$ determines a well-defined point $\bar v\in\mathcal X(K)$.

Conditions~\textnormal{(S1)} and~\textnormal{(S4)} imply that $X_1$ is a compact connected PL $1$-manifold whose boundary consists of its two distinct vertex faces.
Hence $X_1$ is PL homeomorphic to the closed interval.

For every edge $e=[u_0<u_1]$ of $K$, choose a PL homeomorphism $q_e\colon e \to X(e)$ carrying $u_i$ to $\bar u_i$ for $i=0,1$.
These maps agree on common vertices and therefore define a continuous map $s\colon K^{(1)} \to\mathcal X(K)$.

Since $K$ is connected, its $1$-skeleton $K^{(1)}$ is connected, and hence $s( K^{(1)} )$ is path-connected.
Every point $x\in\mathcal X(K)$ lies in the image of some block $X(\sigma)$.
The block $X(\sigma)\cong X_{\dim\sigma}$ is path-connected by condition~\textnormal{(S1)} and contains $\bar w$ for every vertex $w$ of $\sigma$.
Thus $x$ can be joined by a path to a point of $s( K^{(1)} )$.
It follows that $\mathcal X(K)$ is path-connected.

Let $i\colon K^{(1)} \hookrightarrow K$ denote the inclusion.
For every edge $e$ of $K$, both $(F_K\circ s)|_{ e }$ and $i|_{ e }$ take values in $ e $ and agree on the endpoints of $e$.
Since $ e $ is contractible, these two maps are homotopic relative to $\partial e $.
The resulting edgewise homotopies agree on the vertices and therefore give a homotopy $F_K\circ s\simeq i$ relative to $ K^{(0)} $.

After choosing the basepoints $v\in K^{(1)} $ and $\bar v=s(v)$, we obtain $(F_K)_*\circ s_*=i_*$ on fundamental groups.
By cellular approximation, $i_*\colon\pi_1( K^{(1)} ,v)\to\pi_1( K,v)$ is surjective.
Therefore $(F_K)_*$ is surjective.

The assertion for $H_1$ follows because first integral homology is the abelianization of the fundamental group for a path-connected space.
\end{proof}

Since the construction preserves disjoint unions, connectivity questions may be treated separately on each connected component.
Accordingly, whenever fundamental groups are considered below, the relevant ambient complex will be assumed connected.

We next establish homological surjectivity.

\begin{lemma}
\label{lem:homological-surjectivity}
For every $q\geq0$, the structure map induces a split surjection
\[
(F_K)_*\colon
H_q\bigl(\mathcal X(K);\mathbb Z\bigr)
\longrightarrow
H_q(K;\mathbb Z).
\]
\end{lemma}

\begin{proof}
After taking compatible subdivisions, choose a triangulation $T$ of $\mathcal X(K)$ and a subdivision $K'$ of $K$ such that every block $X(\sigma)$ and every distinguished face is a subcomplex of $T$ and the structure map is simplicial,
$
F_K\colon T\longrightarrow K'.
$
Let
$
s\colon
C_*(K;\mathbb Z)
\longrightarrow
C_*(K';\mathbb Z)
$
be the subdivision chain map.

For every oriented $k$-simplex
$
\sigma=[v_0<\cdots<v_k]
$
of $K$, let
$
\langle X(\sigma)\rangle
\in
C_k(T;\mathbb Z)
$
denote the simplicial orientation chain of the corresponding oriented block $X(\sigma)$.

For $k\geq1$, the orientation convention in Definition~\ref{def:oriented-degree-one-system} gives
$
\partial\langle X(\sigma)\rangle
=
\sum_{i=0}^k
(-1)^i
\langle X(\partial_i\sigma)\rangle.
$
For a vertex $v$, define
$
\langle X(v)\rangle:=\bar v,
$
where $\bar v$ is the corresponding vertex of $\mathcal X(K)$. Therefore the assignment
$
j(\sigma)
:=
\langle X(\sigma)\rangle
$
extends linearly to a chain map
$
j\colon
C_*(K;\mathbb Z)
\longrightarrow
C_*(T;\mathbb Z).
$

We claim that
$
F_{K\#}\circ j=s,
$
where
$
F_{K\#}\colon
C_*(T;\mathbb Z)
\longrightarrow
C_*(K';\mathbb Z)
$
is the simplicial chain map induced by $F_K$.

Let $\sigma$ be an oriented $k$-simplex of $K$. The restriction
$
F_K|_{X(\sigma)}
\colon
\bigl(X(\sigma),\partial X(\sigma)\bigr)
\longrightarrow
\bigl(\sigma,\partial\sigma\bigr)
$
has degree $+1$. Hence
$
F_{K\#}\langle X(\sigma)\rangle
$
represents the positive relative fundamental class of
$(\sigma,\partial\sigma)$ in
$
H_k\bigl(
K'|_\sigma,
K'|_{\partial\sigma};
\mathbb Z
\bigr).
$
The subdivision chain $s(\sigma)$ represents the same relative fundamental class.

Since $K'|_\sigma$ has dimension $k$ and $K'|_{\partial\sigma}$ has dimension $k-1$, we obtain
$
C_{k+1}\bigl(
K'|_\sigma,
K'|_{\partial\sigma};
\mathbb Z
\bigr)
=
0
$
and
$
C_k\bigl(
K'|_{\partial\sigma};
\mathbb Z
\bigr)
=
0.
$
It follows that a relative homology class in degree $k$ has at most one simplicial $k$-cycle representative.
Therefore
$
F_{K\#}\langle X(\sigma)\rangle
=
s(\sigma).
$
Since this holds for every oriented simplex $\sigma$, we obtain
$
F_{K\#}\circ j=s.
$

Passing to homology gives
$
F'_{K*}\circ j_*=s_*,
$
where
$
F'_{K*}\colon
H_q(T;\mathbb Z)
\longrightarrow
H_q(K';\mathbb Z)
$
is induced by the simplicial map $F_K\colon T\to K'$.
The subdivision homomorphism
$
s_*\colon
H_q(K;\mathbb Z)
\longrightarrow
H_q(K';\mathbb Z)
$
is an isomorphism.
Under the canonical identifications
$
H_q(T;\mathbb Z)
\cong
H_q(\mathcal X(K);\mathbb Z)
$
and
$
H_q(K';\mathbb Z)
\cong
H_q(K;\mathbb Z)
$
given by triangulation and subdivision invariance, the homomorphism induced by the structure map satisfies
$
(F_K)_*
=
s_*^{-1}\circ F'_{K*}.
$
Consequently,
$
(F_K)_*\circ j_*
=
s_*^{-1}\circ F'_{K*}\circ j_*
=
s_*^{-1}\circ s_*
=
\operatorname{id}_{H_q(K;\mathbb Z)}.
$
Thus $(F_K)_*$ is split surjective in every degree.
\end{proof}

Our next goal is to understand the local behavior of the asphericalization. We first recall a few standard notions from PL topology.
Let $K$ be a simplicial complex, let $a\in K$ be a $k$-simplex, and write $a^\circ$ for its relative interior.
The \emph{link} of $a$ in $K$, denoted $\Lk_K(a)$, is the subcomplex
\[
\Lk_K(a)\;=\;\{\;\tau\in K \mid \tau\cap a=\emptyset,\ \tau\cup a\in K\;\}.
\]
Equivalently, $\Lk_K(a)$ consists of those faces complementary to $a$ inside the star of $a$. If $b$ is an $n$-simplex of $K$ with $a<b$, then $\Lk_b(a)$ is the simplex spanned by the vertices of $b$ not in $a$, hence $\Lk_b(a)\cong \Delta^{\,n-k-1}$.

The \emph{dual cone} of $a$ in $K$, denoted $\mathrm{Dual}_K(a)$, is the cone on $\Lk_K(a)$:
\[
\mathrm{Dual}_K(a)\;=\;\mathrm{Cone}\big(\Lk_K(a)\big).
\]
Its \emph{open dual cone} is defined by $\mathrm{Dual}_K^{\circ}(a):=\mathrm{Dual}_K(a)\setminus \Lk_K(a)$.

If $a$ is a $k$-face of an $n$-simplex $\Delta^n$, then $\mathrm{Dual}_{\Delta^n}(a)\cong \Delta^{\,n-k}$ and $\mathrm{Dual}_{\Delta^n}^{\circ}(a)\cong \mathbb{R}^{\,n-k}_{+}$. In particular, every $k$-simplex $a$ admits a canonical open product neighborhood in $ K$ of the form
\[
a^\circ \times \mathrm{Dual}_K^{\circ}(a).
\]
We will use this product structure repeatedly below.

We now record the corresponding local structure of the direct simplex replacement.

For a simplex $a\in K$, write
\[
X(a)^\circ
:=
X(a)\setminus
\bigcup_{b\subsetneq a}X(b)
\]
for the relative interior of the replacement block $X(a)$.

\begin{lemma}[Local product structure]
\label{lem:local-geometry}
Let $a\in K$ be a simplex.
Then $X(a)^\circ$ admits a PL open neighborhood $U_a$ in $\mathcal X(K)$ together with a PL homeomorphism
\[
U_a
\cong_{\PL}
X(a)^\circ\times\Dual_K^\circ(a).
\]
Moreover, there is a PL open neighborhood $V_a$ of $a^\circ$ in $K$ with
\[
V_a
\cong_{\PL}
a^\circ\times\Dual_K^\circ(a)
\]
such that $F_K(U_a)\subseteq V_a$ and, under these product identifications,
\[
F_K|_{U_a}
=
\left(F_K|_{X(a)^\circ}\right)
\times
\operatorname{id}_{\Dual_K^\circ(a)}.
\]
\end{lemma}

\begin{proof}
Let $\sigma\in K$ be a simplex containing $a$, and put $S:=S_\sigma(a)$.
Under the identification $X(\sigma)\cong X_{\dim\sigma}$, the subspace $X(a)\subseteq X(\sigma)$ is the distinguished face $X_{\dim\sigma}[S]$.
Condition~\textnormal{(S5)} therefore gives a product collar of $X(a)$ in $X(\sigma)$ whose normal coordinates are indexed by the vertices of $\sigma$ not contained in $a$.

Restricting these collars to $X(a)^\circ$, and allowing $\sigma$ to vary over all simplices containing $a$, the compatibility condition in \textnormal{(S5)} shows that the resulting product neighborhoods agree on their common faces.
Hence they glue in $\mathcal X(K)$.
The normal orthants glue according to the cofaces of $a$ in $K$, giving a truncated dual cone of $a$, which is PL homeomorphic to $\Dual_K^\circ(a)$.
Thus they determine a PL open neighborhood $U_a$ of $X(a)^\circ$ with $U_a \cong_{\PL} X(a)^\circ\times\Dual_K^\circ(a)$.

The compatible standard collars $\kappa_S^{\dim\sigma}$ give, in the same way, a PL open neighborhood $V_a$ of $a^\circ$ with $V_a \cong_{\PL} a^\circ\times\Dual_K^\circ(a)$.

By condition~\textnormal{(S6)(c)}, on the collar corresponding to every coface $\sigma\supseteq a$, the structure map preserves the normal coordinates and restricts on the distinguished face $X(a)$ to $F_K|_{X(a)}$.
These identities are compatible on common faces by condition~\textnormal{(S6)(a)}, and hence glue over all cofaces of $a$.
Therefore, in the above product coordinates, $F_K|_{U_a} = \left(F_K|_{X(a)^\circ}\right) \times \operatorname{id}_{\Dual_K^\circ(a)}$.
\end{proof}

The local product structure shows that the direct simplex replacement has the same normal local structure as the original simplicial complex.
We record the following consequences.

\begin{corollary}[Preservation of PL manifolds]
\label{thm:manifold}
If $K$ is a PL $n$-manifold with possibly empty boundary, then $\mathcal X(K)$ is a PL $n$-manifold with possibly empty boundary.
Moreover,
\[
\partial\mathcal X(K)
=
\mathcal X(\partial K)
=
F_K^{-1}(\partial K).
\]
\end{corollary}

\begin{proof}
Let $x\in\mathcal X(K)$.
There is a unique simplex $a\in K$ such that $x\in X(a)^\circ$.
By Lemma~\ref{lem:local-geometry}, a neighborhood of $x$ in $\mathcal X(K)$ is PL homeomorphic to a neighborhood of $(x,0)$ in
$X(a)^\circ\times\Dual_K^\circ(a)$.

Since $X(a)^\circ$ is a PL $\dim a$-manifold without boundary, every point of $X(a)^\circ$ has a neighborhood PL homeomorphic to $\mathbb R^{\dim a}$.
On the other hand, since $K$ is a PL $n$-manifold, Definition~\ref{def:PLmfd} gives
\[
\Lk_K(a)\cong_{\PL}
\begin{cases}
S^{n-\dim a-1}, & a^\circ\subset\operatorname{int}K,\\[2mm]
B^{n-\dim a-1}, & a^\circ\subset\partial K.
\end{cases}
\]
Consequently,
\[
\Dual_K^\circ(a)\cong_{\PL}
\begin{cases}
\mathbb R^{n-\dim a}, & a^\circ\subset\operatorname{int}K,\\[2mm]
\mathbb R^{n-\dim a-1}\times[0,\infty), & a^\circ\subset\partial K.
\end{cases}
\]
It follows that $x$ has a neighborhood PL homeomorphic to $\mathbb R^n$ in the first case and to $\mathbb R^{n-1}\times[0,\infty)$ in the second.
Thus $\mathcal X(K)$ is a PL $n$-manifold, and
$x\in\partial\mathcal X(K)$ if and only if $F_K(x)\in\partial K$. Therefore,
$\partial\mathcal X(K)=F_K^{-1}(\partial K)$.

Since $F_K^{-1}(\partial K)=\mathcal X(\partial K)$ by the exact face-preimage property, the asserted boundary identity follows.
\end{proof}

\begin{corollary}[Preservation of orientation]
\label{thm:oriented}
If $K$ is an oriented PL $n$-manifold, then $\mathcal X(K)$ is naturally oriented.
\end{corollary}

\begin{proof}
For every $n$-simplex $\sigma=[v_0<\cdots<v_n]$ of $K$, let
$\epsilon_\sigma\in\{+1,-1\}$ be the sign such that the orientation of $\sigma$ induced from the given orientation of $K$ is $\epsilon_\sigma$ times its ordered orientation.
Orient the block $X(\sigma)\cong X_n$ by $\epsilon_\sigma$ times the fixed orientation of $X_n$.

We verify that these block orientations agree under the gluings.
Let $\sigma$ and $\sigma'$ be two $n$-simplices meeting along an $(n-1)$-simplex $\tau$.
Write $\tau=\partial_i\sigma=\partial_j\sigma'$ with respect to the ordered vertex sets of $\sigma$ and $\sigma'$.
Since the orientation of $K$ is coherent, the orientations induced on the common face $\tau$ from $\sigma$ and $\sigma'$ are opposite.
That is, $\epsilon_\sigma(-1)^i
=
-\epsilon_{\sigma'}(-1)^j$.

By Definition~\ref{def:oriented-degree-one-system}, the boundary orientation on the distinguished facet of $X(\sigma)$ corresponding to $\tau$ is
$\epsilon_\sigma(-1)^i$ times the fixed orientation of $X_{n-1}$, whereas the boundary orientation on the corresponding facet of $X(\sigma')$ is
$\epsilon_{\sigma'}(-1)^j$ times the fixed orientation of $X_{n-1}$.
The preceding identity therefore shows that these two induced boundary orientations are opposite on their common copy $X(\tau)$. Hence the orientations of the top-dimensional blocks glue across every interior facet and determine a global orientation of $\mathcal X(K)$.
\end{proof}

We now establish the global asphericity property of the construction.
We use the standard graph-of-spaces gluing theorem: if a space is obtained by gluing aspherical vertex spaces along aspherical edge spaces whose attaching maps are $\pi_1$-injective, then every path component of the resulting space is aspherical, and the inclusions of the vertex spaces induce monomorphisms on fundamental groups; see, for example,~\cite[pp.~156--157]{Scott-Wall:1979}.
When an edge space is disconnected, we regard each of its path components as a separate edge space.

\begin{theorem}[Asphericity and $\pi_1$-injectivity]
\label{thm:aspherical-manifold}
Suppose that $\mathcal X_{\leq n}$ is an asphericalized simplex system.
Then, for every subcomplex $L\subseteq K$, each path component of $\mathcal X(L)$ is aspherical, and its inclusion into $\mathcal X(K)$ induces a monomorphism on fundamental groups.
In particular, $\mathcal X(K)$ is aspherical.

Consequently, if $K$ is a PL $n$-manifold, then $\mathcal X(K)$ is an aspherical PL $n$-manifold.
If $K$ is oriented, then $\mathcal X(K)$ is naturally oriented.
\end{theorem}

\begin{proof}
We prove simultaneously, by induction on the number of simplices of the ambient simplicial complex, that the direct replacement is componentwise aspherical and that the replacement of every subcomplex is componentwise $\pi_1$-injective in it.

The assertion is immediate for a $0$-dimensional complex.
Assume that it holds for all simplicial complexes having fewer simplices than $K$.

We first prove that $\mathcal X(K)$ is componentwise aspherical.
Choose a maximal simplex $\sigma\in K$, say $\dim\sigma=k$, and set $K':=K\setminus\{\sigma\}$.
Since $\sigma$ is maximal, $K'$ is a subcomplex of $K$ and $K'\cap\sigma=\partial\sigma$.
By Remark~\ref{rem:replacement-gluing}, $\mathcal X(K) = \mathcal X(K') \cup_{\mathcal X(\partial\sigma)} X(\sigma)$.

By the induction hypothesis, every path component of $\mathcal X(K')$ is aspherical, and the inclusion $\mathcal X(\partial\sigma) \longrightarrow \mathcal X(K')$ is componentwise $\pi_1$-injective.
On the other hand, under the identification $X(\sigma)\cong X_k$, one has $\mathcal X(\partial\sigma) = X_k[\partial\Delta^k]$.
Condition~\textnormal{(A)}, applied to $P=\Delta^k$ and $P=\partial\Delta^k$, implies that $X(\sigma)$ and every path component of $\mathcal X(\partial\sigma)$ are aspherical and that $\mathcal X(\partial\sigma) \longrightarrow X(\sigma)$ is componentwise $\pi_1$-injective.
The graph-of-spaces gluing theorem therefore implies that every path component of $\mathcal X(K)$ is aspherical and that $\mathcal X(K') \longrightarrow \mathcal X(K)$ is componentwise $\pi_1$-injective.

It remains to prove the assertion for an arbitrary proper subcomplex $L\subsetneq K$.
Choose a maximal simplex $\sigma$ of $K$ which is not contained in $L$.
Such a simplex exists because $K$ is finite and $L$ is a subcomplex.
Again set $K':=K\setminus\{\sigma\}$.
Then $L\subseteq K'$.
By the induction hypothesis, the inclusion $\mathcal X(L) \longrightarrow \mathcal X(K')$ is componentwise $\pi_1$-injective.
By the preceding gluing argument, $\mathcal X(K') \longrightarrow \mathcal X(K)$ is also componentwise $\pi_1$-injective.
Their composite therefore gives a monomorphism $\pi_1(C) \longrightarrow \pi_1(\mathcal X(K))$ for every path component $C$ of $\mathcal X(L)$.

Finally, since $L$ itself has fewer simplices than $K$ whenever $L\subsetneq K$, the induction hypothesis also shows that every path component of $\mathcal X(L)$ is aspherical.
This completes the induction.

The final assertions follow from Corollaries~\ref{thm:manifold} and~\ref{thm:oriented}.
\end{proof}

\subsection{Relative asphericalization}

We end this section by introducing a relative version of the asphericalization.

Let $J\subseteq K$ be a subcomplex, and let $J=J_1\sqcup\cdots\sqcup J_r$ be the decomposition of $J$ into its path components.
For each $i$, introduce a new vertex $c_i$ and attach the cone $c_i*J_i$ to $K$ along its base $J_i$.
Thus we form the simplicial complex $K\cup CJ := K \cup_{J_1}(c_1*J_1) \cup\cdots\cup_{J_r}(c_r*J_r)$.
We extend the ordering of the vertices of $K$ to $K\cup CJ$, for example by requiring $c_1<\cdots<c_r<v$ for every vertex $v$ of $K$.

Consider the direct asphericalization $\mathcal X(K\cup CJ)$ and, for each $i$, let $\bar c_i:=X(c_i)\in\mathcal X(K\cup CJ)$ be the point corresponding to the cone vertex $c_i$.
Since $\Lk_{K\cup CJ}(c_i)=J_i$, Lemma~\ref{lem:local-geometry} gives a PL open neighborhood $U_i$ of $\bar c_i$ with $U_i \cong_{\PL} \Dual_{K\cup CJ}^{\circ}(c_i) \cong_{\PL} \operatorname{Cone}(J_i)\setminus J_i$.
Moreover, under this identification the structure map is the identity in the conical coordinate.

Choose pairwise disjoint closed PL neighborhoods $C_i\subset U_i$ of the points $\bar c_i$, each corresponding to a closed truncated cone on $J_i$.
Then $C_i\cong_{\PL}\operatorname{Cone}(J_i)$ and there is a natural PL identification $\theta_i\colon \partial C_i \xrightarrow{\cong_{\PL}} J_i$.

\begin{definition}[Relative asphericalization]\label{def:relative-asphericalization}
The \emph{relative asphericalization of $K$ with respect to $J$} is defined by
\[
\mathcal X(K,J)
:=
\mathcal X(K\cup CJ)
\setminus
\bigcup_{i=1}^r\operatorname{int}C_i.
\]
Via the identifications $\theta_i\colon\partial C_i\cong_{\PL}J_i$, we regard each $J_i$ as a distinguished subspace of $\mathcal X(K,J)$.
\end{definition}

\begin{remark}
\label{rem:relative-asphericalization-choice}
The PL type of $\mathcal X(K,J)$ does not depend on the choice of the sufficiently small truncated conical neighborhoods $C_i$.
Indeed, in the fixed conical coordinates of Lemma~\ref{lem:local-geometry}, any two such choices are related by a radial PL isotopy supported in $U_i$ and equal to the identity outside $U_i$.
Since the neighborhoods $U_i$ are pairwise disjoint, these isotopies may be performed simultaneously, giving a PL homeomorphism between the resulting relative asphericalizations which is compatible with the canonical boundary identifications $\partial C_i\cong_{\PL}J_i$.
\end{remark}

Thus relative asphericalization is obtained by first coning off each component of $J$, applying the ordinary asphericalization to the resulting cone-off, and then removing a small open conical neighborhood of each point lying over a cone vertex.
Equivalently, $\partial C_i\cong_{\PL}J_i$ is the link of the deleted conical end in the resulting relative space.

\begin{remark}
\label{rem:relative-manifold}
Suppose that $K$ is a PL $n$-manifold and that $J=J_1\sqcup\cdots\sqcup J_r$ is a union of connected components of $\partial K$.
Then $(K\cup CJ)\setminus\{c_1,\ldots,c_r\}$ is a PL $n$-manifold.

In fact, away from $J$ this is immediate.
Near a point of $J_i$, a PL collar $J_i\times[0,1) \subseteq K$ and the punctured cone $(c_i*J_i)\setminus\{c_i\} \cong_{\PL} J_i\times(0,1]$ glue along $J_i$ to give a PL product neighborhood with a full normal coordinate.
Thus the only possible singular points of $K\cup CJ$ are the cone vertices $c_i$.

By Lemma~\ref{lem:local-geometry}, the same conclusion holds after direct asphericalization.
That is, $\mathcal X(K\cup CJ) \setminus \{\bar c_1,\ldots,\bar c_r\}$ is a PL $n$-manifold.
Removing the interiors of the conical neighborhoods $C_i$ therefore gives a compact PL $n$-manifold $\mathcal X(K,J)$.
Moreover, each $\partial C_i\cong_{\PL}J_i$ is a boundary component of $\mathcal X(K,J)$.

If $J^{c}:= \overline{\partial K\setminus J}$ denotes the union of the remaining boundary components of $K$, then
\[
\partial\mathcal X(K,J)
=
\left(
\bigsqcup_{i=1}^r\partial C_i
\right)
\sqcup
\mathcal X(J^{c}),
\]
and hence, after identifying $\partial C_i$ with $J_i$, $\partial\mathcal X(K,J) \cong_{\PL} J\sqcup\mathcal X(J^{c})$.
In particular, if $J=\partial K$, then $\partial\mathcal X(K,\partial K) \cong_{\PL} \partial K$.
\end{remark}

\begin{remark}
\label{rem:relative-orientation}
If $K$ is oriented and $J$ is a union of boundary components, then $\mathcal X(K,J)$ carries a natural orientation.
In fact, the orientation of $K$ extends across $(c_i*J_i)\setminus\{c_i\}$ so that the two induced orientations along the common copy of $J_i$ are opposite.
Every top-dimensional simplex $\sigma$ of $K\cup CJ$ has interior contained in the oriented manifold $(K\cup CJ)\setminus\{c_1,\ldots,c_r\}$, and orienting the corresponding replacement block $X(\sigma)$ by the sign of this simplex orientation makes the block orientations agree across every common codimension-one distinguished face by the same facet-orientation calculation as in the proof of Corollary~\ref{thm:oriented}.
Hence $\mathcal X(K\cup CJ)\setminus\{\bar c_1,\ldots,\bar c_r\}$ is oriented, and deleting the interiors of the conical neighborhoods $C_i$ gives the asserted orientation on $\mathcal X(K,J)$.
\end{remark}

\begin{example}
Let $M$ be a closed PL $(n-1)$-manifold and set $K=M\times[0,1]$ and $J=M\times\{0,1\}$.
Then $\mathcal X\bigl(M\times[0,1],\,M\times\{0,1\}\bigr)$ is a compact PL $n$-manifold whose two boundary components are naturally identified with copies of $M$.
Thus it is a PL cobordism from $M$ to $M$.
\end{example}

%%%%%%%%%%%%%%%%%%%%%%%%%%%%%%%%%%%
%%%%%%%%%Section 3%%%%%%%%%%%%%%%%%
%%%%%%%%%%%%%%%%%%%%%%%%%%%%%%%%%%%

\section{Toroidal asphericalized simplices}
\label{sec:toroidal-simplex}

In this section, we construct an explicit asphericalized simplex system, called the system of \emph{toroidal simplices}, which will serve as the basic building block for the quantitative bordisms developed below.
The construction is modeled on Gromov's cylinder construction as formulated by Davis--Januszkiewicz~\cite[Section~4c]{Davis-Januszkiewicz:1991-1}, but is adapted to the direct simplex replacement of Section~\ref{sec:asph}.

We first recall the reflection-cylinder construction.

\begin{definition}[PL reflection]
Let $Y$ be a PL space and let $r\colon Y\to Y$ be a PL involution.
We call $r$ a \emph{PL reflection} if there is a PL subspace $A\subseteq Y$ such that, with $B:=A\cap r(A)$, one has $B=\operatorname{Fix}(r)$ and the natural map
\[
D(A,B):=
\bigl(A\times\{-1,+1\}\bigr)/{\sim}
\longrightarrow Y,
\]
where $(b,-1)\sim (b,+1)$ for every $b\in B$, defined by
\[
[a,+1]\longmapsto a,
\qquad
[a,-1]\longmapsto r(a),
\]
is a PL homeomorphism.
The subspace $A$ is called a \emph{half-space} for $r$.

If $Y$ is a PL manifold and the half-space $A$ is a PL manifold with boundary $B=\operatorname{Fix}(r)$, then, following Davis--Januszkiewicz~\cite[Section~4c]{Davis-Januszkiewicz:1991-1}, we call $r$ a \emph{locally linear PL reflection}.
\end{definition}

Let $r$ be a PL reflection on $Y$ with half-space $A$.
Following~\cite[Section~4c]{Davis-Januszkiewicz:1991-1}, define
\[
\Omega(Y,A,r)
:=
\bigl(Y\times[-1,1]\bigr)/{\sim},
\]
where
\[
(y,-1)\sim(y,1)
\qquad
\text{for every }y\in r(A).
\]
Let
\[
q\colon
Y\times[-1,1]
\longrightarrow
\Omega(Y,A,r)
\]
be the quotient map.
The boundary of $\Omega(Y,A,r)$ is the image of $A\times\{-1,+1\}$ and is naturally identified with $D(A,B)$, hence with $Y$.

We will repeatedly use the following consequence of the reflection-cylinder construction.

\begin{lemma}[Reflection cylinder]
\label{lem:reflection-cylinder}
Let $Y$ be a closed PL $n$-manifold and let $r$ be a locally linear PL reflection on $Y$ with half-space $A$.
Then $\Omega(Y,A,r)$ is a compact PL $(n+1)$-manifold with
\[
\partial\Omega(Y,A,r)\cong_{\PL}Y.
\]
If $Y$ is aspherical, then $\Omega(Y,A,r)$ is aspherical and the boundary inclusion induces a monomorphism
\[
\pi_1(Y)
\longrightarrow
\pi_1\bigl(\Omega(Y,A,r)\bigr).
\]
\end{lemma}

\begin{proof}
This is the PL and asphericity part of~\cite[Proposition~4c.2]{Davis-Januszkiewicz:1991-1}.
\end{proof}

We now incorporate the reflection into our coherent simplex system.
For $k\geq1$, let
\[
\rho_k\colon\Delta^k\longrightarrow\Delta^k
\]
be the simplicial involution induced by the transposition $(0\,1)$, that is,
\[
\rho_k(e_0)=e_1,
\qquad
\rho_k(e_1)=e_0,
\qquad
\rho_k(e_i)=e_i
\quad
(i\geq2),
\]
and set $\rho_0=\operatorname{id}_{\Delta^0}$.

Suppose that a coherent PL simplex system $\mathcal X_{\leq n}=\{(X_k,f_k)\}_{k=0}^n$ has been constructed together with PL reflections $r_k\colon X_k\longrightarrow X_k$, where $0\leq k\leq n$, and chosen half-spaces $H_k\subseteq X_k$, satisfying $f_k\circ r_k=\rho_k\circ f_k$.
For $k\geq1$, we furthermore require that, near every point of $\operatorname{Fix}(r_k)$, the pair $(X_k,H_k)$ has the standard PL half-space model in which $r_k$ acts by $(z,t)\mapsto(z,-t)$ and $H_k$ corresponds to $t\geq0$, with the model taken productwise with the appropriate boundary orthant at points of $\partial X_k$.

We require the reflections to be compatible with all distinguished faces as follows.
For a nonempty subset $S\subseteq[k]$, let $\rho_{k,S}\colon [|S|-1] \longrightarrow [|\rho_k(S)|-1]$ be the unique bijection satisfying $\rho_k\circ\iota_S = \iota_{\rho_k(S)}\circ\rho_{k,S}$.
Since $\rho_k$ is induced by the transposition $(0\,1)$, one has
\[
\rho_{k,S}
=
\begin{cases}
(0\,1), & \{0,1\}\subseteq S,\\
\operatorname{id}, & \{0,1\}\nsubseteq S.
\end{cases}
\]
Accordingly, define
\[
\epsilon_S
:=
\begin{cases}
r_{|S|-1}, & \{0,1\}\subseteq S,\\
\operatorname{id}_{X_{|S|-1}},
& \{0,1\}\nsubseteq S.
\end{cases}
\]
We require $r_k\circ j_S^k = j_{\rho_k(S)}^k\circ\epsilon_S$ for every nonempty $S\subseteq[k]$.

The product collars in condition~\textnormal{(S5)} are chosen equivariantly as well.
More precisely, if
\[
u=(u_a)_{a\in[k]\setminus S}
\in
[0,\varepsilon)^{[k]\setminus S},
\]
write $\rho_k u\in
[0,\varepsilon)^{[k]\setminus\rho_k(S)}$ for the vector defined by $(\rho_k u)_{\rho_k(a)}=u_a$.
Then we require
\[
r_k\bigl(c_S^k(x,u)\bigr)
=
c_{\rho_k(S)}^k
\bigl(\epsilon_S(x),\rho_k u\bigr).
\]

We furthermore choose the half-space $H_k$ compatibly with the distinguished facets.
Writing
\[
F_i^k
:=
X_k\bigl[[k]\setminus\{i\}\bigr],
\]
we require
\[
H_k\cap F_0^k=F_0^k,
\qquad
H_k\cap F_1^k=F_0^k\cap F_1^k,
\]
and, for $2\leq i\leq k$,
\[
H_k\cap F_i^k
=
j_{[k]\setminus\{i\}}^k(H_{k-1}),
\]
where $F_i^k$ is identified with $X_{k-1}$ by
$j_{[k]\setminus\{i\}}^k$.

For $n=0$, set
\[
X_0:=\Delta^0,
\qquad
f_0:=\operatorname{id}_{\Delta^0},
\qquad
r_0:=\operatorname{id}_{\Delta^0},
\qquad
H_0:=X_0.
\]
For $n=1$, set
\[
X_1:=\Delta^1,
\qquad
f_1:=\operatorname{id}_{\Delta^1},
\qquad
r_1:=\rho_1,
\]
and let $H_1$ be the closed half-interval joining $e_1$ to the midpoint of $\Delta^1$.
The required local half-space model is immediate for $(X_1,H_1,r_1)$, while $r_0$ is used only as the degenerate zero-dimensional initial datum.

Assume inductively that a reflection-compatible oriented degree-one asphericalized simplex system $\mathcal X_{\leq n} = \{(X_k,f_k)\}_{k=0}^n$ has been constructed.
Apply the direct simplex replacement to the boundary of $\Delta^{n+1}$ and set
\[
Y_n
:=
\mathcal X_{\leq n}\bigl(\partial\Delta^{n+1}\bigr),
\qquad
g_n
:=
F_{\partial\Delta^{n+1}}
\colon
Y_n
\longrightarrow
\partial\Delta^{n+1}.
\]
By Theorem~\ref{thm:aspherical-manifold}, $Y_n$ is aspherical.
By Corollaries~\ref{thm:manifold} and~\ref{thm:oriented}, it is a closed oriented PL $n$-manifold.

For $0\leq i\leq n+1$, let
\[
\sigma_i
:=
\Delta^{n+1}\bigl[[n+1]\setminus\{i\}\bigr]
\]
be the $i$-th facet of $\Delta^{n+1}$, and let
\[
Y_{n,i}
:=
\mathcal X_{\leq n}(\sigma_i)
\subseteq
Y_n.
\]
Thus each $Y_{n,i}$ is canonically identified with a copy of $X_n$.

The involution $\rho_{n+1}$ exchanges $\sigma_0$ and $\sigma_1$ and preserves each $\sigma_i$ for $i\geq2$.
Using the face-compatibility formula above, define $R_n\colon Y_n\longrightarrow Y_n$ blockwise as follows.
The map $R_n$ exchanges $Y_{n,0}$ and $Y_{n,1}$ by the canonical identification of their model blocks, while, for every $i\geq2$, its restriction to $Y_{n,i}\cong X_n$ is $r_n$.
The compatibility formula $r_k\circ j_S^k = j_{\rho_k(S)}^k\circ\epsilon_S$ shows that these blockwise maps agree on all common distinguished faces.
Hence they descend to a well-defined PL involution $R_n\colon Y_n\longrightarrow Y_n$.
By construction, $g_n\circ R_n = \rho_{n+1}\circ g_n$.

For $i\geq2$, let $H_{n,i}\subseteq Y_{n,i}$ be the copy of the chosen half-space $H_n\subseteq X_n$ under the canonical identification $Y_{n,i}\cong X_n$, and define
\[
A_n
:=
Y_{n,0}
\cup
\bigcup_{i=2}^{n+1}H_{n,i}.
\]
The compatibility conditions imposed on the half-spaces $H_k$ imply that these pieces agree on their common distinguished faces, so $A_n$ is a well-defined PL subspace of $Y_n$.

Since $R_n$ exchanges $Y_{n,0}$ and $Y_{n,1}$ and exchanges the two half-spaces of every invariant block $Y_{n,i}$, $i\geq2$, one has $Y_n = A_n\cup R_n(A_n)$.
Moreover, $A_n\cap R_n(A_n) = \operatorname{Fix}(R_n)$.
In fact, on the exchanged pair $Y_{n,0}\cup Y_{n,1}$ the fixed points are precisely those lying in their common distinguished faces fixed by the induced labeled involution, while on each invariant block $Y_{n,i}$, $i\geq2$, the assertion follows from $H_n\cap r_n(H_n) = \operatorname{Fix}(r_n)$.
These descriptions agree on all intersections by the inductive face-compatibility condition.
It follows that the natural map $D\bigl(A_n,\operatorname{Fix}(R_n)\bigr) \longrightarrow Y_n$ is a PL homeomorphism.
Thus $A_n$ is a half-space for $R_n$.

It remains to verify that $A_n$ is a PL $n$-manifold with boundary $\operatorname{Fix}(R_n)$.
Since
\[
Y_n\cong_{\PL}D\bigl(A_n,\operatorname{Fix}(R_n)\bigr),
\]
every point of $A_n\setminus\operatorname{Fix}(R_n)$ has a PL manifold neighborhood in $A_n$, so it suffices to consider points of $\operatorname{Fix}(R_n)$.
Let $x\in\operatorname{Fix}(R_n)$, and let $S$ be the smallest distinguished face whose relative interior contains $x$.
If $S$ contained exactly one of $0$ and $1$, then $\rho_{n+1}(S)\neq S$, and $R_n$ would carry the relative interior of the $S$-stratum to the disjoint relative interior of the $\rho_{n+1}(S)$-stratum; hence such a stratum contains no fixed point.
Thus either $\{0,1\}\subseteq S$ or $\{0,1\}\cap S=\emptyset$.
If $\{0,1\}\subseteq S$, then the inductive local half-space model for the corresponding lower-dimensional reflection, together with the equivariant product collars, gives a PL half-space neighborhood of $x$ in $A_n$ whose boundary is $\operatorname{Fix}(R_n)$.
If $\{0,1\}\cap S=\emptyset$, then the compatible normal orthants glue to the full dual-cone neighborhood, $R_n$ interchanges the coordinates $u_0$ and $u_1$, and after setting $s=u_0+u_1$ and $t=u_0-u_1$, the involution is $(z,s,t)\mapsto(z,s,-t)$ while, after possibly replacing $t$ by $-t$, $A_n$ corresponds locally to $t\geq0$.
Indeed, in these coordinates the two local chambers exchanged by $R_n$ are separated by the fixed hyperplane $t=0$, and taking the product with the remaining normal coordinates and the stratum identifies the pair $\bigl(Y_n,A_n\bigr)$ locally, preserving the labeled coordinate strata, with the standard PL reflection pair $\bigl(\mathbb R^{n-1}\times\mathbb R,\mathbb R^{n-1}\times[0,\infty)\bigr)$.
Hence $A_n$ is a PL $n$-manifold with boundary $\operatorname{Fix}(R_n)$.
Since $Y_n$ is a closed PL $n$-manifold, $R_n$ is therefore a locally linear PL reflection.
After an equivariant subdivision, we may moreover assume that $R_n$ is simplicial and $A_n$ is a subcomplex.

Define $X_{n+1} := \Omega(Y_n,A_n,R_n)$.
By Lemma~\ref{lem:reflection-cylinder}, $X_{n+1}$ is a compact PL $(n+1)$-manifold and
\[
\partial X_{n+1}
\cong_{\PL}
Y_n
=
\mathcal X_{\leq n}\bigl(\partial\Delta^{n+1}\bigr).
\]
Since $Y_n$ is oriented, the product orientation on $Y_n\times[-1,1]$ descends through the reflection-cylinder identification: the two identified portions of the end slices carry opposite boundary orientations.
Moreover, the locally linear reflection $R_n$ is orientation reversing, so the canonical identification $\partial X_{n+1}\cong_{\PL}Y_n$ may be taken orientation preserving.
We orient $X_{n+1}$ accordingly.

Let $q\colon Y_n\times[-1,1] \longrightarrow X_{n+1} = \Omega(Y_n,A_n,R_n)$ be the quotient map.
The involution $(y,t)\longmapsto(y,-t)$ preserves the equivalence relation defining the reflection cylinder and therefore descends to a PL involution $r_{n+1}\colon X_{n+1} \longrightarrow X_{n+1}$.
Set $H_{n+1} := q\bigl(Y_n\times[0,1]\bigr)$.
Then $X_{n+1} = H_{n+1}\cup r_{n+1}(H_{n+1})$, and $H_{n+1}\cap r_{n+1}(H_{n+1}) = \operatorname{Fix}(r_{n+1})$.
In fact, the fixed set consists of $q\bigl(Y_n\times\{0\}\bigr)$ together with the points of the two end slices identified by the reflection-cylinder relation.
Consequently, $D\bigl(H_{n+1},\operatorname{Fix}(r_{n+1})\bigr)\longrightarrow X_{n+1}$ is a PL homeomorphism.

The quotient charts of the reflection-cylinder construction show that near the fixed set $r_{n+1}$ has the local form $(z,t)\longmapsto(z,-t)$.
Hence $r_{n+1}$ is a locally linear PL reflection with half-space $H_{n+1}$.

Under the canonical identification $\partial X_{n+1}\cong_{\PL}Y_n$, the restriction of $r_{n+1}$ to the boundary is precisely $R_n$.
It follows from the construction of $A_n$ that $r_{n+1}$ satisfies the same face-compatibility and half-space-incidence conditions as the reflections $r_k$, $k\leq n$.
Thus the reflection data extend from $\mathcal X_{\leq n}$ to $\mathcal X_{\leq n+1}$.

Through this identification, the blocks of $Y_n$ corresponding to the proper faces of $\Delta^{n+1}$ define the distinguished faces of $X_{n+1}$.
For each proper nonempty subset $S\subsetneq[n+1]$, the distinguished face $X_{n+1}[S]\subseteq\partial X_{n+1}=Y_n$ is identified with $X_{|S|-1}$, and the inductively chosen collar structure on $Y_n$ identifies its normal germ with the boundary of the standard orthant $[0,\varepsilon)^{[n+1]\setminus S}$, with the coordinate faces carrying the prescribed vertex labels.
Combining this normal model with a product collar of $Y_n=\partial X_{n+1}$ in $X_{n+1}$ gives a collar $c_S^{n+1}\colon X_{|S|-1}\times[0,\varepsilon)^{[n+1]\setminus S}\to X_{n+1}$ of $X_{n+1}[S]$.
Here the identification of the boundary-orthant normal model times the boundary-collar coordinate with the full orthant is chosen label-preservingly and inductively over the face poset, relative to all coordinate faces already fixed in lower dimensions; such a choice exists by the relative PL Alexander trick, and therefore its restriction to every coordinate suborthant is exactly the previously chosen identification.
These orthant identifications are chosen $\mathbb Z/2$-equivariantly, so that
\[
r_{n+1}\bigl(c_S^{n+1}(x,u)\bigr)
=
c_{\rho_{n+1}(S)}^{n+1}
\bigl(\epsilon_S(x),\rho_{n+1}u\bigr);
\]
hence the collar-equivariance condition is preserved in dimension $n+1$.
For $\emptyset\neq T\subseteq S$, the nested-face identity for $c_T^{n+1}$ and $c_S^{n+1}$ now follows exactly from the inductive identity on $Y_n$ together with the preceding relative choice of the orthant identifications.
Thus the face embeddings, their intersections, and these collars satisfy conditions~\textnormal{(S3)--(S5)} for $X_{n+1}$.

We next define the structure map $f_{n+1}\colon X_{n+1}\longrightarrow\Delta^{n+1}$.
Choose compatible collars
\[
c\colon
Y_n\times[0,\varepsilon)
\longrightarrow
X_{n+1}
\qquad\text{and}\qquad
\kappa\colon
\partial\Delta^{n+1}\times[0,\varepsilon)
\longrightarrow
\Delta^{n+1}
\]
so that $r_{n+1}\bigl(c(y,u)\bigr) = c\bigl(R_n(y),u\bigr)$ and $\rho_{n+1}\bigl(\kappa(z,u)\bigr) = \kappa\bigl(\rho_{n+1}(z),u\bigr)$.
On the collar, define $f_{n+1}\bigl(c(y,u)\bigr) := \kappa\bigl(g_n(y),u\bigr)$.
Since $g_n\circ R_n = \rho_{n+1}\circ g_n$, this prescribed collar map is $\mathbb Z/2$-equivariant.

We now extend it equivariantly over the interior.
Choose $0<\delta<\varepsilon$ and set
\[
D
:=
\overline{
X_{n+1}
\setminus
c\bigl(Y_n\times[0,\delta)\bigr)
}.
\]
Then $\partial D = c\bigl(Y_n\times\{\delta\}\bigr)$.
Since $\delta>0$, the prescribed map on $\partial D=c\bigl(Y_n\times\{\delta\}\bigr)$ takes values in $\operatorname{int}\Delta^{n+1}$.
Choose a triangulation of $\partial D$ on which this prescribed map is affine on every simplex, and extend it to a triangulation of $D$.
Choose a point $b\in\operatorname{int}\Delta^{n+1}$,
for instance the barycenter.
Define a map on the vertices of this triangulation by retaining the prescribed values on the vertices of $\partial D$ and sending every interior vertex of $D$ to $b$, and extend affinely over every simplex.
This gives a PL map $h\colon D\longrightarrow \operatorname{int}\Delta^{n+1}$ extending the prescribed map on $\partial D$.
In fact, the image of every vertex lies in $\operatorname{int}\Delta^{n+1}$, and since $\operatorname{int}\Delta^{n+1}$ is convex, the affine image of every simplex is contained in $\operatorname{int}\Delta^{n+1}$.
On each simplex of $\partial D$, the resulting affine map agrees with the prescribed map because the two maps agree on its vertices.
Together with the prescribed map on the boundary collar, this gives a PL map $h\colon X_{n+1}\longrightarrow\Delta^{n+1}$ satisfying $h\bigl(X_{n+1}\setminus\partial X_{n+1}\bigr) \subseteq \operatorname{int}\Delta^{n+1}$.

Take a common $r_{n+1}$-invariant subdivision on which both $h$ and $\rho_{n+1}\circ h\circ r_{n+1}$ are PL, and define
\[
f_{n+1}(x)
:=
\frac{1}{2}
\left(
h(x)
+
\rho_{n+1}\bigl(h(r_{n+1}(x))\bigr)
\right),
\]
where $\Delta^{n+1}$ is regarded as the standard convex simplex in Euclidean space.
Since $\Delta^{n+1}$ is convex, this defines a PL map $f_{n+1}\colon X_{n+1} \longrightarrow \Delta^{n+1}$.
Moreover, since $\operatorname{int}\Delta^{n+1}$ is convex, we obtain $f_{n+1} \bigl(X_{n+1}\setminus\partial X_{n+1}\bigr) \subseteq \operatorname{int}\Delta^{n+1}$.

On the chosen boundary collar the map $h$ is already equivariant.
Hence, for $x$ in this collar,
\[
\rho_{n+1}\bigl(h(r_{n+1}(x))\bigr)
=
h(x),
\]
so the symmetrization does not alter the prescribed collar map.
Finally, using $r_{n+1}^2=\operatorname{id}$ and $\rho_{n+1}^2=\operatorname{id}$, we obtain $f_{n+1}\circ r_{n+1} = \rho_{n+1}\circ f_{n+1}$.
Thus $f_{n+1}$ is a $\mathbb Z/2$-equivariant PL extension of the prescribed structure map on the boundary collar.

Since $f_{n+1}|_{\partial X_{n+1}}=g_n$, its restriction to every distinguished face agrees with the corresponding lower-dimensional structure map, and therefore condition~\textnormal{(S6)(a)} holds.
The exact face-preimage property follows immediately: for every proper nonempty subset $S\subsetneq[n+1]$, $f_{n+1}^{-1}\bigl(\Delta^{n+1}[S]\bigr) = X_{n+1}[S]$.
The product formula in the chosen collars gives condition~\textnormal{(S6)(c)}.

The orientations are chosen so that the identification $\partial X_{n+1}\cong Y_n$ preserves orientation.
With the boundary orientations induced from $\Delta^{n+1}$ and $X_{n+1}$, the fundamental chains are respectively $\sum_{i=0}^{n+1}(-1)^i[\sigma_i]$ and $\sum_{i=0}^{n+1}(-1)^i\langle X(\sigma_i)\rangle$.
Since each block map $f_n\colon X(\sigma_i)\to\sigma_i$ has degree $+1$, the induced map $(g_n)_\#$ sends the latter fundamental chain to the former.
Hence $(g_n)_*[Y_n]=[\partial\Delta^{n+1}]$, and therefore $\deg(g_n)=+1$.
By naturality of the boundary homomorphism for the pairs $(X_{n+1},\partial X_{n+1})$ and $(\Delta^{n+1},\partial\Delta^{n+1})$, it follows that
\[
f_{n+1}\colon
(X_{n+1},\partial X_{n+1})
\longrightarrow
(\Delta^{n+1},\partial\Delta^{n+1})
\]
also has degree $+1$.

It remains to verify condition~\textnormal{(A)}.

\begin{theorem}
\label{thm:toroidal-asphericalized-simplex}
For every $n\geq0$, the inductive construction above produces an asphericalized simplex system $\mathcal X_{\leq n} = \{(X_k,f_k)\}_{k=0}^n$.
In particular, $(X_n,f_n)$ is an asphericalized $n$-simplex.
\end{theorem}

\begin{proof}
We proceed by induction on $n$.
The assertion is immediate for $n=0,1$.

Assume that $\mathcal X_{\leq n}$ is an asphericalized simplex system.
By Theorem~\ref{thm:aspherical-manifold},
\[
Y_n
=
\mathcal X_{\leq n}\bigl(\partial\Delta^{n+1}\bigr)
\]
is aspherical.
Lemma~\ref{lem:reflection-cylinder} therefore implies that $X_{n+1}$ is aspherical and that the boundary inclusion
\[
Y_n
=
\partial X_{n+1}
\longrightarrow
X_{n+1}
\]
induces a monomorphism on fundamental groups.

Let $P\subseteq\Delta^{n+1}$ be a simplicial subcomplex.
If $P=\Delta^{n+1}$, then
\[
X_{n+1}[P]=X_{n+1},
\]
which is aspherical.

Suppose that $P\subsetneq\Delta^{n+1}$.
Then $P\subseteq\partial\Delta^{n+1}$, and by the definition of the distinguished faces,
\[
X_{n+1}[P]
=
\mathcal X_{\leq n}(P)
\subseteq
Y_n.
\]
By Theorem~\ref{thm:aspherical-manifold}, every path component of $\mathcal X_{\leq n}(P)$ is aspherical and its inclusion into $Y_n$ induces a monomorphism on fundamental groups.
Composing with the monomorphism
\[
\pi_1(Y_n)
\longrightarrow
\pi_1(X_{n+1})
\]
given by Lemma~\ref{lem:reflection-cylinder}, we conclude that every path component of $X_{n+1}[P]$ is aspherical and its inclusion into $X_{n+1}$ is $\pi_1$-injective.
Thus condition~\textnormal{(A)} holds for $X_{n+1}$.
\end{proof}

We call the asphericalized $n$-simplex constructed above the \emph{toroidal $n$-simplex}.

\begin{remark}
For $n=2$, the construction starts with $Y_1 = \mathcal X_{\leq1}(\partial\Delta^2) \cong S^1$.
The reflection cylinder $\Omega(Y_1,A_1,R_1)$ is a compact orientable surface with one boundary component and Euler characteristic $-1$.
Hence $X_2 \cong_{\PL} T^2\setminus\operatorname{int}D^2$.
Thus the first nontrivial member of the construction is precisely the punctured torus appearing in Remark~\ref{rem:low-dimensional-examples}.
Moreover, $Y_2 = \mathcal X_{\leq2}(\partial\Delta^3)$ is a closed orientable surface of genus $4$.
\end{remark}

%%%%%%%%%%%%%%%%%%%%%%%%%%%%%%%%%%%
%%%%%%%%%Section 4%%%%%%%%%%%%%%%%%
%%%%%%%%%%%%%%%%%%%%%%%%%%%%%%%%%%%

\section{Quantitative relative asphericalization}
\label{sec:rel-asph}

In this section, we apply the toroidal simplex system constructed in Section~\ref{sec:toroidal-simplex} to obtain quantitative relative asphericalizations and the bordisms used in the proof of Theorem~\ref{thm:quantitative-bordism}.

We first introduce the asphericalization groups associated with a triangulated manifold equipped with a faithful representation and an asphericalized simplex system.

\begin{definition}[Asphericalization group]
\label{def:asph-gr}
Let $G$ be a group, let $\mathcal X$ be an asphericalized simplex system of length at least $n+1$, and let $M$ be a connected PL $n$-manifold equipped with an ordered triangulation $K$ and a faithful representation $\varphi\colon\pi_1(M)\longrightarrow G$.
Equip $M\times I$ with the standard staircase triangulation induced by the ordering of the vertices of $K$.
The \emph{asphericalization group of $G$ associated with $(M,K,\mathcal X,\varphi)$} is defined by
\[
\Asph_{M,K,\mathcal X,\varphi}(G)
:=
G
\ast_{\pi_1(M)}
\pi_1\!\left(
\mathcal X(M\times I,M_+\sqcup M_-)\cup CM_-
\right),
\]
where the two homomorphisms defining the amalgamated product are $\varphi$ and the homomorphism induced by the canonical inclusion
\[
M_+
\longrightarrow
\mathcal X(M\times I,M_+\sqcup M_-)\cup CM_-.
\]
Thus the ordered triangulation $K$, and hence the induced staircase triangulation of $M\times I$, is part of the defining data of $\Asph_{M,K,\mathcal X,\varphi}(G)$.
When $K$, $\mathcal X$, and $\varphi$ are understood from the context, we write simply $\Asph_M(G)$.
\end{definition}

We next establish the quantitative estimate for the toroidal simplices that will be used throughout the remainder of the paper.

For each $n\geq1$, choose a triangulation $T_n$ of $X_n$ inductively with the following properties:
every distinguished face is a subcomplex with the triangulation inherited from the corresponding lower-dimensional toroidal simplex, the reflection $r_n$ is simplicial, and a chosen half-space for $r_n$ is a subcomplex.
We choose these triangulations inductively, compatibly with the reflection-cylinder construction and the local product coordinates of Lemma~\ref{lem:local-geometry}, so that the closed star of every distinguished vertex is a simplicial conical neighborhood in those coordinates.
Let $d_n$ denote the number of $n$-simplices of $T_n$.

\begin{lemma}
\label{lem:quantitative-Xn}
For every $n\geq1$, the toroidal $n$-simplex satisfies
\[
\Delta(X_n)
\leq
4^{n-1} \cdot n! \cdot (n+1)!.
\]
\end{lemma}

\begin{proof}
For $n=1$, let $T_1$ be the subdivision of the interval $X_1=\Delta^1$ at the fixed point of the reflection $r_1$.
Then the two reflection half-spaces are subcomplexes and $d_1=2$.

Suppose now that $n\geq2$ and that $T_{n-1}$ has been constructed.
Recall that $Y_{n-1} = \mathcal X_{\leq n-1}(\partial\Delta^n)$.
The boundary $\partial\Delta^n$ has $n+1$ facets, and the replacement of each facet is a copy of $X_{n-1}$.
Since distinct facet blocks meet only along lower-dimensional distinguished faces, the triangulation induced on $Y_{n-1}$ contains exactly $(n+1)d_{n-1}$ simplices of dimension $n-1$.

By the inductive compatibility of the triangulations with the reflections and their half-spaces, the reflection $R_{n-1}$ on $Y_{n-1}$ is simplicial and its chosen half-space $A_{n-1}$ is a subcomplex.

Recall that $X_n = \Omega(Y_{n-1},A_{n-1},R_{n-1}) = \bigl(Y_{n-1}\times[-1,1]\bigr)/{\sim}$.
Subdivide $[-1,1]$ at $-1, -\frac12, 0, \frac12, 1$.
Thus the interval is divided into four subintervals, symmetrically with respect to $t\mapsto-t$.

For every $(n-1)$-simplex $\tau$ of $Y_{n-1}$ and every one of these four subintervals $I$, triangulate the prism $\tau\times I$ by the standard staircase triangulation into $n$ $n$-simplices.
Choose the triangulations on opposite subintervals to be exchanged by $t\mapsto-t$.
Consequently, the reflection on the cylinder is simplicial.

The equivalence relation defining the reflection cylinder identifies subcomplexes of the two end slices.
Indeed, the quotient identifies the two copies of the simplicial subcomplex $R_{n-1}(A_{n-1})$ in the end slices by the simplicial identity map; since every simplex of the staircase triangulation meets an end slice in a possibly empty face and no simplex meets both end slices, no two distinct vertices of a simplex are identified, and the images of any two simplices intersect, if at all, in the image of a common face.
It follows that the quotient is an ordinary simplicial complex, hence gives a simplicial triangulation of $X_n$; moreover, the quotient introduces no new $n$-simplices and therefore does not increase their number.
By the inductive compatibility of $T_{n-1}$ and by our symmetric choice of the staircase triangulations, the quotient triangulation restricts on every distinguished face to the prescribed lower-dimensional triangulation, makes $r_n$ simplicial and $H_n$ a subcomplex, and carries the closed stars of the distinguished vertices to the simplicial conical neighborhoods determined by the inductive local product coordinates; thus no further subdivision is required.
Then, $d_n \leq 4n(n+1)d_{n-1}$. 
Iterating this inequality and using $d_1=2$, we obtain $d_n \leq 4^{n-1}n!(n+1)!$.
Since $\Delta(X_n)$ is the minimum number of $n$-simplices among all PL triangulations of $X_n$, while $T_n$ is one particular triangulation, we have $\Delta(X_n)\leq d_n$.
Therefore, $\Delta(X_n) \leq 4^{n-1}n!(n+1)!$.
\end{proof}

We now establish the group-theoretic property of the relative asphericalization that will be used in the bordism construction.

\begin{lemma}[Peripheral $\pi_1$-injectivity]
\label{lem:peripheral-pi1}
Let $\mathcal X_{\leq n}$ be an asphericalized simplex system, let $K$ be a compact connected PL $n$-manifold, and let
\[
J=J_1\sqcup\cdots\sqcup J_r\subseteq\partial K
\]
be a union of connected boundary components.
For $I\subseteq\{1,\ldots,r\}$, set
\[
W_I
:=
\mathcal X(K,J)
\cup_{\bigsqcup_{i\in I}J_i}
\bigsqcup_{i\in I}CJ_i.
\]
Then, for every $j\notin I$, the inclusion $J_j\longrightarrow W_I$ induces a monomorphism on fundamental groups.
\end{lemma}

\begin{proof}
For each $i$, set $Z_i := \mathcal X(c_i*J_i) \subseteq \mathcal X(K\cup CJ)$ and $T_i := Z_i\setminus\operatorname{int}C_i$.
Since $c_i*J_i$ is connected, Lemma~\ref{lem:pi1-surj} shows that $Z_i$ is path-connected. Then, every point $x\in T_i$ can be joined in $Z_i$ to $\bar c_i\in\operatorname{int}C_i$.
Truncating such a path at its first entrance into $\operatorname{int}C_i$ gives a path in $T_i$ from $x$ to $\partial C_i$.
Since $\partial C_i\cong_{\PL}J_i$ is path-connected, it follows that $T_i$ is path-connected.
By Remark~\ref{rem:replacement-gluing},
\[
\mathcal X(K,J)
=
\mathcal X(K)
\cup_{\bigsqcup_{i=1}^r\mathcal X(J_i)}
\bigsqcup_{i=1}^r T_i.
\]

We first record two $\pi_1$-injectivity properties of $T_i$.

Let $F_i := F_{c_i*J_i} \colon Z_i \longrightarrow c_i*J_i$ be the structure map.
By the exact face-preimage property, $F_i^{-1}(c_i)=\{\bar c_i\}$.
Since $\bar c_i\in\operatorname{int}C_i$, the restriction of $F_i$ to $T_i$ takes values in $(c_i*J_i)\setminus\{c_i\}$.
Let $p_i\colon (c_i*J_i)\setminus\{c_i\} \longrightarrow J_i$ be the radial projection of the punctured cone onto its base.
By the choice of $C_i$ in the local conical coordinates of Lemma~\ref{lem:local-geometry}, one has $p_i\circ F_i|_{\partial C_i} = \theta_i$.
Hence, after identifying $\partial C_i$ with $J_i$ by $\theta_i$, the map $r_i := p_i\circ F_i|_{T_i} \colon T_i \longrightarrow J_i$ restricts to the identity on $J_i$.
Thus $J_i$ is a retract of $T_i$, and consequently $\pi_1(J_i) \longrightarrow \pi_1(T_i)$ is injective.

Next, since $C_i$ is disjoint from the base $\mathcal X(J_i)\subseteq Z_i$, there are inclusions $\mathcal X(J_i) \longrightarrow T_i \longrightarrow Z_i$.
Their composite is the canonical inclusion $\mathcal X(J_i) \longrightarrow \mathcal X(c_i*J_i)$.
Since $J_i\subseteq c_i*J_i$ is a simplicial subcomplex, Theorem~\ref{thm:aspherical-manifold} implies that this composite induces a monomorphism on fundamental groups.
Therefore $\pi_1\bigl(\mathcal X(J_i)\bigr)\longrightarrow \pi_1(T_i)$ is injective.

Now fix $I\subseteq\{1,\ldots,r\}$.
For each $i\in I$, coning off the boundary component $J_i$ restores the deleted conical neighborhood.
In fact,
\[
T_i\cup_{J_i}CJ_i
\cong_{\PL}
T_i\cup_{\partial C_i}C_i
=
Z_i,
\]
where the first PL homeomorphism is relative to the common boundary $J_i\cong_{\PL}\partial C_i$.

Define
\[
K_I
:=
K
\cup_{J_i,\ i\in I}
\bigcup_{i\in I}(c_i*J_i).
\]
By the pushout description of the direct simplex replacement,
\[
\mathcal X(K_I)
=
\mathcal X(K)
\cup_{\bigsqcup_{i\in I}\mathcal X(J_i)}
\bigsqcup_{i\in I}Z_i.
\]
Consequently,
\[
W_I
\cong_{\PL}
\mathcal X(K_I)
\cup_{\bigsqcup_{j\notin I}\mathcal X(J_j)}
\bigsqcup_{j\notin I}T_j.
\]

This is a finite tree of spaces with central vertex space $\mathcal X(K_I)$, leaf vertex spaces $T_j$ for $j\notin I$, and edge spaces $\mathcal X(J_j)$.

For every $j\notin I$, the inclusion $\mathcal X(J_j) \longrightarrow \mathcal X(K_I)$ is $\pi_1$-injective by Theorem~\ref{thm:aspherical-manifold}, since
$J_j\subseteq K_I$ is a subcomplex.
We proved above that $\mathcal X(J_j) \longrightarrow T_j$ is also $\pi_1$-injective.
Therefore, by successive applications of the Seifert--van Kampen theorem and the injectivity theorem for graphs of groups~\cite[Theorem~1B.11 and Example~1B.12]{Hatcher:2002-1}, the canonical map $\pi_1(T_j) \longrightarrow \pi_1(W_I)$ is injective.

Finally, we have already shown that $\pi_1(J_j) \longrightarrow \pi_1(T_j)$ is injective.
Hence the composite $\pi_1(J_j) \longrightarrow \pi_1(T_j) \longrightarrow \pi_1(W_I)$ is injective, as required.
\end{proof}

\begin{lemma}
\label{lem:geometric}
Let $M$ be a closed connected PL $n$-manifold.
There exists a compact PL cobordism $W$ from $M_+$ to $M_-$ over a quotient group $\G$ of $\pi_1(W)$ such that:
\begin{enumerate}
\item
the composition
$\pi_1(M_-)
\longrightarrow
\pi_1(W)
\longrightarrow
\G$
is trivial;

\item
the composition
$\pi_1(M_+)
\longrightarrow
\pi_1(W)
\longrightarrow
\G$
is injective;

\item
the complexity satisfies $\Delta(W) \le C(n) \cdot \Delta(M)$, where $C(n) = (n+3) \cdot 4^n \cdot (n+1)! \cdot (n+2)!$.
\end{enumerate}
\end{lemma}

\begin{proof}
Choose an ordered triangulation $K$ of $M$ with $f_n(K)=\Delta(M)$, put $N:=\Delta(M)$, and consider the toroidal asphericalized simplex system $\mathcal X$ of Theorem~\ref{thm:toroidal-asphericalized-simplex}.
Put $P:=M\times[0,1]$, $M_-:=M\times\{0\}$, and $M_+:=M\times\{1\}$.
Triangulate $P$ by the standard staircase triangulation associated with the chosen ordering of the vertices of $K$.

Apply the relative toroidal asphericalization to the pair $\bigl(P,M_-\sqcup M_+\bigr)$ using the truncation $\mathcal X_{\leq n+1}$ of the toroidal asphericalized simplex system, and set $W:=\mathcal X\bigl(P,M_-\sqcup M_+\bigr)$.
Since $M_-\sqcup M_+ = \partial P$, Remark~\ref{rem:relative-manifold} shows that $W$ is a compact PL $(n+1)$-manifold with $\partial W \cong_{\PL} M_- \sqcup M_+$.
Thus $W$ is a PL cobordism from $M$ to $M$.

Cone off the incoming boundary and set $W' := W\cup_{M_-}CM_-$.
Define $\G:=\pi_1(W')$.
By the Seifert--van Kampen theorem, $\G \cong \pi_1(W) \Big/ \left\langle\!\left\langle \operatorname{im} \bigl( \pi_1(M_-)\longrightarrow\pi_1(W) \bigr) \right\rangle\!\right\rangle$.
Hence the composition $\pi_1(M_-) \longrightarrow \pi_1(W) \longrightarrow \G$ is trivial.

Applying Lemma~\ref{lem:peripheral-pi1} with $J_1=M_-$, $J_2=M_+$, and $I=\{1\}$, gives a monomorphism
\[
\pi_1(M_+)
\longrightarrow
\pi_1\bigl(W\cup_{M_-}CM_-\bigr)
=
\G.
\]
This proves the first two assertions.

It remains to estimate the complexity of $W$.
For every $n$-simplex $\sigma$ of $K$, the standard staircase triangulation decomposes $\sigma\times[0,1]$ into exactly $n+1$ simplices of dimension $n+1$.
So, $\Delta(P) \leq (n+1)\Delta(M)$.

In the relative construction we first form the cone-off
\[
\widehat P
:=
P
\cup_{M_-}(c_-*M_-)
\cup_{M_+}(c_+*M_+).
\]
The staircase triangulation of $P=M\times[0,1]$ contains exactly $(n+1)N$ simplices of dimension $n+1$, while each cone $c_\pm*M_\pm$ contains exactly $N$ such simplices.
Hence $f_{n+1}(\widehat P)=(n+3)N$.
Now apply the direct toroidal asphericalization to $\widehat P$, and let $d_{n+1}$ denote the number of $(n+1)$-simplices in the compatible triangulation $T_{n+1}$ of $X_{n+1}$ used in Lemma~\ref{lem:quantitative-Xn}.
Since distinct top-dimensional replacement blocks meet only along lower-dimensional distinguished faces, the induced triangulation satisfies $f_{n+1}\bigl(\mathcal X(\widehat P)\bigr)=(n+3)d_{n+1}N$.

The points $\bar c_\pm=X(c_\pm)$ are vertices of the induced triangulation, and we take $C_\pm:=\operatorname{St}(\bar c_\pm)$.
By the compatible local product coordinates of Lemma~\ref{lem:local-geometry}, each $C_\pm$ is a simplicial cone with $\partial C_\pm=\Lk(\bar c_\pm)\cong_{\PL}M_\pm$, and this identification agrees with the radial cone coordinate used in the relative construction.
Moreover, $C_+$ and $C_-$ are disjoint by construction.
Thus $C_\pm$ are admissible conical neighborhoods in the definition of the relative asphericalization, and deleting their interiors requires no further subdivision and introduces no new top-dimensional simplices.
Since removing the interiors of $C_\pm$ introduces no new simplices, the resulting triangulation of $W$ has at most $(n+3)d_{n+1}N$ simplices of dimension $n+1$.
Therefore, by Lemma~\ref{lem:quantitative-Xn},
\[
\Delta(W)
\le
(n+3)\, \cdot 4^n\, \cdot (n+1)! \cdot (n+2)!\, \cdot \Delta(M).
\]
\end{proof}

We finally prove Theorem~\ref{thm:quantitative-bordism}.

\begin{proof}[Proof of Theorem~\ref{thm:quantitative-bordism}]
Let $W=\mathcal X\bigl(M\times[0,1],M_-\sqcup M_+\bigr)$ be the cobordism given by Lemma~\ref{lem:geometric}, let $i_\pm\colon M_\pm\longrightarrow W$ denote the boundary inclusions, and let ${i_W}_*  \colon\pi_1(W)\longrightarrow\mathcal G$ be the quotient representation constructed there.
Set $\psi:={i_W}_* \circ(i_+)_*\colon\pi_1(M_+)\longrightarrow\mathcal G$; by Lemma~\ref{lem:geometric}, $\psi$ is faithful, whereas ${i_W}_*\circ(i_-)_*$ is trivial.
Using the canonical identification $\pi_1(M_+)\cong\pi_1(M)$, form the amalgamated product $\Gamma:=G*_{\pi_1(M_+)}\mathcal G$ with respect to the faithful representations $\varphi$ and $\psi$.
By the normal form theorem for amalgamated free products, the canonical homomorphisms $i_G\colon G\longrightarrow\Gamma$ and $i_{\mathcal G}\colon\mathcal G\longrightarrow\Gamma$ are injective, and by Definition~\ref{def:asph-gr} this group is precisely $\Asph_{M,K,\mathcal X,\varphi}(G)$.
Define a representation $\Phi\colon\pi_1(W)\longrightarrow\Gamma$ by $\Phi:=i_{\mathcal G}\circ {i_W}_*$.
By the defining relation of the amalgamated product, $\Phi\circ(i_+)_*=i_{\mathcal G}\circ\psi=i_G\circ\varphi$, while $\Phi\circ(i_-)_*=1$.
Thus $\Phi$ defines a bordism over $\Gamma$ whose restriction to $M_+$ is the given representation $\varphi$ followed by the inclusion $G\hookrightarrow\Gamma$, and whose restriction to $M_-$ is trivial:
 \[
\adjustbox{scale=1, center}{%
\begin{tikzcd}
\pi_{1}(M_{+}) \arrow[d, hook, "i_{+*}"] \arrow[rd, hook, "i_{W*} \circ i_{+*}"] \arrow[r, hook, "\varphi" ]  & G \arrow[rd, hook, "i_G"]  & \\
\pi_{1}(W) \arrow[r, "i_{W*}"] & \G = \pi_{1}(W \cup C M_{-}) \arrow[r, hook, "i_{\G}"] & \Gamma := G \ast_{\pi_1(M_{+})} \G = \textup{Asph}_M(G) \\
\pi_{1}(M_{-}) \arrow[u, hook, "i_{-*}"] \arrow[ru, "\times 0"']
& &
\end{tikzcd}
}
\]
Moreover, $\Delta(W)\leq C(n)\Delta(M)$, where $C(n)=O\!\bigl(n\cdot4^n\cdot(n+1)!\cdot(n+2)!\bigr)$.
If $M$ is oriented, then $M\times[0,1]$ carries the product orientation, and Remark~\ref{rem:relative-orientation} gives $W=\mathcal X\bigl(M\times[0,1],M_-\sqcup M_+\bigr)$ its natural orientation.
\end{proof}

%%%%%%%%%%%%%%%%%%%%%%%%%%%%%%%%%%%
%%%%%%%%%Section 5%%%%%%%%%%%%%%%%%
%%%%%%%%%%%%%%%%%%%%%%%%%%%%%%%%%%%

\section{Enhanced bounds for Cheeger--Gromov rho-invariants}
\label{sec:linear-rho}

In this section, we explain how the existence of a linear bordism established in Theorem~\ref{thm:quantitative-bordism} leads to the linear bound on the Cheeger--Gromov $L^2$ $\rho$-invariant $\rho^{(2)}(M)$ stated in Theorem~\ref{thm:efficient-rho}.  
The argument follows the approach of~\cite[Section~2]{Cha:2014-1} and~\cite[Section~6]{Cha-Lim:2024-1}.

\begin{proof}[Proof of Theorem~\ref{thm:efficient-rho}]

Consider a closed, oriented PL manifold $M^{4k-1}$ together with a faithful representation $\varphi \colon \pi_{1}(M) \to G$.  
Since the $L^2$ $\rho$-invariant and the simplicial complexity behave additively under disjoint unions, one may assume $M$ is connected.  
Theorem~\ref{thm:quantitative-bordism} ensures the existence of a bordism $W$ over $\Gamma$ from $(M_{+},\varphi)$ to a trivial end $(M_{-}, *)$.
Furthermore, the construction satisfies the quantitative control
\[
\Delta(W) \le C(n)\, \cdot \Delta(M),
\qquad C(n) \in O\!\bigl(n \cdot 4^{n} \cdot (n+1)! \cdot (n+2)!\bigr),
\]
and the inclusions of the boundary components induce compatible maps on fundamental groups forming a commutative diagram with $\pi_{1}(W)$ and $\Gamma$:
 \[
\adjustbox{scale=1, center}{%
\begin{tikzcd}
\pi_{1}(M_{+}) \arrow[d, hook, "i_{+*}"] \arrow[rd, hook, "i_{W*} \circ i_{+*}"] \arrow[r, hook, "\varphi" ]  & G \arrow[rd, hook, "i_G"]  & \\
\pi_{1}(W) \arrow[r, "i_{W*}"] & \G = \pi_{1}(W \cup C M_{-}) \arrow[r, hook, "i_{\G}"] & \Gamma := G \ast_{\pi_1(M_{+})} \G = \textup{Asph}_M(G) \\
\pi_{1}(M_{-}) \arrow[u, hook, "i_{-*}"] \arrow[ru, "\times 0"']
& &
\end{tikzcd}
}
\]

By the $L^2$-induction property (see, for instance,~\cite[Eq.~(2.3)]{Cheeger-Gromov:1985-1},~\cite[p.~253]{Lueck:2002-1},~\cite[Prop.~5.13]{Cochran-Orr-Teichner:1999-1}), both faithful representations $\varphi$ and $i_G\circ\varphi$ give the same $L^2$ $\rho$-invariant.
On the other hand, the representation of $\pi_1(M_-)$ in $\Gamma$ is trivial, so its $L^2$ $\rho$-invariant is zero.
Therefore the bordism formula gives
\[
\rho^{(2)}(M)
=
\operatorname{sign}^{(2)}_{\Gamma}(W)
-
\operatorname{sign}(W).
\]

The ordinary and $L^2$-signatures satisfy
\[
\bigl|\sign(W)\bigr|\leq b_{2k}(W)
\qquad\text{and}\qquad
\bigl|\sign^{(2)}_{\Gamma}(W)\bigr|\leq b^{(2)}_{2k}(W).
\]
Choose a PL triangulation $T$ of $W$ with $f_{4k}(T)=\Delta(W)$.
Since both the ordinary and $L^2$-Betti numbers are bounded by the number of $2k$-simplices of $T$, we obtain
\[
\bigl|\rho^{(2)}(M)\bigr|
\leq
b^{(2)}_{2k}(W)+b_{2k}(W)
\leq
2\cdot f_{2k}(T).
\]
Since $T$ is a pure $4k$-dimensional simplicial complex, every $2k$-simplex of $T$ is a face of a $4k$-simplex, while every $4k$-simplex has exactly $\binom{4k+1}{2k+1}$ faces of dimension $2k$.
Consequently,
\[
f_{2k}(T)
\leq
\binom{4k+1}{2k+1}\cdot f_{4k}(T)
=
\binom{4k+1}{2k+1}\cdot\Delta(W).
\]
Since $n=4k-1$, Lemma~\ref{lem:geometric} therefore gives the explicit estimate
\[
\bigl|\rho^{(2)}(M)\bigr|
\leq
2\cdot
\binom{n+2}{(n+3)/2}
\cdot
(n+3)
\cdot
4^n
\cdot
(n+1)!
\cdot
(n+2)!
\cdot
\Delta(M).
\]
Finally, the central binomial coefficient estimate gives $\binom{n+2}{(n+3)/2}=O\!\bigl(2^n/\sqrt{n+1}\bigr)$, and hence
\[
\bigl|\rho^{(2)}(M)\bigr|
\leq
C_{\mathrm{new}}(n)\cdot\Delta(M),
\qquad
C_{\mathrm{new}}(n)
=
O\!\left(
\frac{1}{\sqrt{n+1}}
\cdot
n
\cdot
8^n
\cdot
(n+1)!
\cdot
(n+2)!
\right).
\]
\qedhere
\end{proof}

%%%%%%%%%%%%%%%%%%%%%%%%%%%%%%%%%%%
%%%%%%%%%Section 6%%%%%%%%%%%%%%%%%
%%%%%%%%%%%%%%%%%%%%%%%%%%%%%%%%%%%

\section{Smooth analogues}
\label{sec:smooth}

Gromov's original linearity problem is formulated in the smooth category, where the complexity of a smooth manifold is measured by Riemannian volume under bounded-geometry assumptions.
By contrast, the results of this paper are formulated in the PL category and use the simplicial complexity $\Delta(-)$.
It is therefore natural to ask whether the quantitative asphericalization developed here admits a smooth analogue with comparable control on complexity.

There are at least two natural approaches to a smooth version of the present method.

\smallskip
\noindent
\textit{(1) Direct smooth asphericalization.}
One could seek a coherent system of smooth manifolds with faces $\mathcal X^{\mathrm{sm}}_{\leq n} = \{(X_k^{\mathrm{sm}},f_k^{\mathrm{sm}})\}_{k=0}^n$ satisfying smooth analogues of the face-coherence, collar, degree-one, and asphericity conditions of Section~\ref{sec:asph}, together with smooth reflection data corresponding to the toroidal construction of Section~\ref{sec:toroidal-simplex}.
The resulting blocks would then be glued directly along their smooth faces.

For a quantitative theory, however, topological asphericity alone would not suffice.
One would need to choose Riemannian metrics on the blocks with uniformly controlled local geometry and arrange the metrics and their normal derivatives compatibly along all distinguished faces.
The smoothing of corners produced by the gluings would also have to preserve these bounds.
Obtaining such a system while retaining effective control of the dimension-dependent constants is a separate problem.
A successful construction would provide a smooth counterpart of the toroidal asphericalization, in the spirit of smooth hyperbolization constructions such as~\cite{Davis-Januszkiewicz:1991-1,Ontaneda:2020}.

\smallskip
\noindent
\textit{(2) Quantitatively smoothing the PL construction.}
A different approach would begin with a smooth manifold $(M,g)$ of bounded local geometry, triangulate it quantitatively, apply the PL asphericalization developed in this paper, and then attempt to smooth the resulting PL bordism.

More precisely, quantitative triangulation results such as~\cite[Theorem~3]{Boissonnat-Dyer-Ghosh:2018-1} produce, under bounded geometry, triangulations $K$ for which the number of top-dimensional simplices is bounded by
\[
\#K^{(n)}
\leq
C_1(n) \cdot \operatorname{vol}(M,g),
\]
together with suitable local combinatorial control.
Applying the present PL construction would then produce a PL bordism whose number of simplices is bounded linearly in $\#K^{(n)}$.

The difficulty lies in the passage back to the smooth category.
First, one must know that the resulting PL bordism admits a smoothing compatible with the prescribed smooth structure on its boundary.
Second, for a quantitative conclusion, one would need such a smoothing with bounded local geometry and with volume controlled by the simplicial complexity of the PL bordism.
The smoothing theory of Hirsch--Mazur and the quantitative PL-to-smooth methods used in~\cite{Chambers-Dotterrer-Manin-Weinberger:2018-1, Manin-Weinberger:2023-1,Cha-Lim:2024-1} suggest a possible framework for this approach.

In particular, even if one obtained, for each fixed $n$, inequalities
of the form
\[
V(W)
\leq
C_{\mathrm{smooth}}(n) \cdot \Delta(W)
\leq
C_{\mathrm{smooth}}(n) \cdot C_{\mathrm{PL}}(n) \cdot \Delta(M)
\leq
C'(n) \cdot V(M),
\]
this would establish linearity in each fixed dimension but would not by itself preserve the factorial asymptotic growth of the constant proved in the PL category.
For that stronger conclusion, one would additionally need sufficiently sharp control of the growth of $C_{\mathrm{smooth}}(n)$ as $n\to\infty$.

Thus developing either a direct smooth toroidal asphericalization or a quantitative smoothing procedure for the PL construction, with effective control of the resulting dimension-dependent constants, appears to be a natural direction for further investigation.

%%%%%%%%%%%%%%%%%%%%%%%5
%%%%%%%%%%%%%%%%%%%%%%%%
%%%%%%%%%%%%%%%%%%%%%%%%

\appendix

\section{Quantitative comparison with the previous bordism constructions}
\label{app:quantitative-comparison}

The purpose of this appendix is to justify the quantitative comparison stated in the Introduction.
We measure the PL complexity of an $n$-manifold throughout by the minimum number of $n$-simplices in a PL triangulation.
In particular, we do not obtain the estimate for the acyclic-group construction by first using the total-simplices convention of~\cite{Cha-Lim:2024-1} and then converting conventions; instead, we trace the construction of~\cite{Cha-Lim:2024-1} from an input triangulation with a prescribed number of top-dimensional simplices.

\subsection{The quantitative Baumslag--Dyer--Heller chain homotopy}

We begin with the algebraic input to the construction of~\cite{Cha-Lim:2024-1}.
The improved chain null-homotopy of~\cite[Theorem~4.3]{Lim:2022-1} is controlled in degree $m$ by the sequence $c(m)$ determined by $c(0)=0$ and
\[
c(m)
=
2^{m+1}-1
+
\sum_{k=1}^{m-1}
c(k)
\binom{m+1}{m-k}.
\]
The estimate obtained in~\cite{Lim:2022-1} gives
\[
c(m)
=
O\!\left(
\left(\frac{2}{e}\right)^m
\cdot
m^{m+3/2}
\right),
\]
and hence, for $m\geq1$,
\[
\log c(m)
=
O\!\bigl(m\log(m+2)\bigr).
\]

The chain homotopy in~\cite[Theorem~4.3]{Lim:2022-1} has Moore chains as its source, whereas the skeleton-lowering construction of~\cite[Proposition~4.2]{Cha-Lim:2024-1} is formulated on the cellular chain complex of the simplicial classifying space.
As in the proof of~\cite[Corollary~3.2]{Cha-Lim:2024-1}, let $g\colon C_*(BG)\longrightarrow\mathbb ZBG_*$ be the standard chain-homotopy inverse to the normalization projection.
In dimensions at most $n$, the construction used there gives $\|g\| \leq (2n+3)^{n+1}$.
Thus, after composing the chain homotopy of~\cite{Lim:2022-1} with $g$, we may use the uniform bound
\[
Q_n
:=
(2n+3)^{n+1}
\cdot
\max_{0\leq m\leq n}c(m)
\]
for the partial cellular chain null-homotopies appearing in the construction of~\cite{Cha-Lim:2024-1}.
Consequently, for $n\geq1$,
\[
\log Q_n
=
O\!\bigl(n\log(n+2)\bigr).
\]

\subsection{The Cha--Lim skeleton-lowering construction}

We now recall the construction of~\cite[Proposition~4.2]{Cha-Lim:2024-1} and trace its complexity in terms of top-dimensional simplices.
Fix $1\leq p\leq n$, let $M_p$ be a closed triangulated $n$-manifold, and suppose that $\xi_p\colon M_p\longrightarrow L$ is a simplicial-cellular map to a simplicial-cell complex $L$ of dimension $p$, where $L$ is a subcomplex of another simplicial-cell complex $K$.
Suppose that $P\colon C_*(L)\longrightarrow C_{*+1}(K)$ is a partial chain null-homotopy in dimensions $p$ and $p-1$ satisfying $P\partial+\partial P=i$ where $i\colon C_*(L)\longrightarrow C_*(K)$ is induced by the inclusion $L\subseteq K$, and $\|P\|\leq q$ where $\|P\|$ denotes the $\ell^1$ operator norm with respect to the simplex bases.
Write $N_p:=f_n(M_p)$ for the number of $n$-simplices in the given triangulation of $M_p$.

\subsubsection*{The inverse-image manifolds $Y_\sigma$.}
For every $p$-simplex $\sigma$ of $L$, let $\widehat\sigma$ denote its barycenter and set $Y_\sigma := \xi_p^{-1}(\widehat\sigma)$.
The simplicial-cellular transversality theorem of~\cite{Cha-Lim:2024-1} shows that $Y_\sigma$ is a closed PL $(n-p)$-submanifold and has a product regular neighborhood $Y_\sigma\times\sigma_0 \subset M_p$, where $\sigma_0\subset\operatorname{int}\sigma$ is the distinguished $p$-simplex neighborhood of $\widehat\sigma$.

More explicitly, if a simplex $A\subset M_p$ maps onto $\sigma$ and the vertices of $A$ are partitioned according to the vertices of $\sigma$, then $A=A_0*\cdots*A_p$ and $Y_\sigma\cap A \cong A_0\times\cdots\times A_p$.
In particular, the cell structure on $\bigcup_\sigma Y_\sigma$ used in~\cite{Cha-Lim:2024-1} has at most one cell associated with each simplex of the source triangulation, and each such cell has at most $(n+1)^{n+1}$ vertices.

\subsubsection*{The inverse-image cobordisms $Z_\tau$.}
For every $(p-1)$-simplex $\tau$ of $L$, Cha and the author construct a properly embedded PL $(n-p+1)$-submanifold $Z_\tau$ whose boundary is expressed in terms of parallel copies of the $Y_\sigma$.
To construct it, each $p$-simplex $\sigma$ of $L$ is first subdivided by inserting the simplex $\sigma_0$ around its barycenter and triangulating the complement of $\operatorname{int}\sigma_0$ by a pulling triangulation.
The complement of $\operatorname{int}\sigma_0$ consists of $p+1$ $p$-cells, each having at most $2p$ vertices.
Thus the number of top-dimensional simplices in this fixed subdivision of a $p$-simplex is at most
$
D_p
:=
1
+
(p+1)\cdot 2^{2p}$.

The corresponding subdivision is then pulled back to $M_p$.
If $B$ is an $n$-simplex of $M_p$ mapping onto a $p$-simplex and $\theta$ is a top-dimensional simplex in the preceding subdivision, the pullback cell $\eta_B^{-1}(\theta)$ is a convex $n$-cell defined by at most $n+p+2$ affine inequalities.
The proof of the quantitative subdivision lemma in~\cite{Cha-Lim:2024-1} therefore gives at most $2^{n+p+2}$ vertices for this cell.
Since the pulling triangulation introduces no new vertices, the number of $n$-simplices in the triangulation of this cell is at most $\binom{2^{n+p+2}}{n+1}$.
Consequently, if
$s_{n,p}
:=
D_p
\cdot
\binom{2^{n+p+2}}{n+1}$,
then the subdivision of $M_p$ used in the construction of the $Z_\tau$ has at most $s_{n,p}\cdot N_p$ top-dimensional simplices.
Uniformly for $1\leq p\leq n$, we have
$
\log s_{n,p}
=
O(n^2)$.

Let $u_n:=2^{n+1}-1$.
Since every simplex of a pure $n$-dimensional simplicial complex is a face of an $n$-simplex, a triangulation with $N$ top-dimensional simplices has at most $u_n\cdot N$ simplices altogether.
It follows that the cell structures on $\bigcup_\sigma Y_\sigma$, on the exterior of their product neighborhoods, and on the original union $\bigcup_\tau Z_\tau$ may all be taken to have at most $u_n \cdot s_{n,p} \cdot N_p$ cells.

In the oriented case, if
$
\partial\sigma
=
\sum_\tau d_{\tau\sigma}\tau,
$
the $Z_\tau$ satisfy
$
\partial Z_\tau
=
\bigcup_\sigma d_{\tau\sigma}\cdot Y_\sigma
$
algebraically.
The initial geometric boundary may contain additional oppositely oriented parallel copies of some $Y_\sigma$, and~\cite{Cha-Lim:2024-1} removes each canceling pair by adjoining a copy of $Y_\sigma\times I$.
As observed in the proof of~\cite[Proposition~4.2]{Cha-Lim:2024-1}, the cells contributed by these additional cylinders are at most three times the number of cells already present in the corresponding $Z_\tau$.
Hence the modified union of the $Z_\tau$ has at most $4 \cdot u_n \cdot s_{n,p} \cdot N_p$ cells.

\subsubsection*{The algebraic bordisms $V_\sigma$.}
For a $p$-simplex $\sigma$ of $L$, write
\[
\partial\sigma
=
\sum_\tau d_{\tau\sigma}\tau,
\qquad
P\sigma
=
\sum_\mu r_{\mu\sigma}\mu,
\qquad
P\tau
=
\sum_\eta s_{\eta\tau}\eta.
\]
Evaluating the chain-homotopy identity $P\partial+\partial P-i=0$ on $\sigma$ gives
\[
\left(
\sum_{\eta,\tau}
s_{\eta\tau}d_{\tau\sigma}\cdot\eta
+
\sum_\mu
r_{\mu\sigma}\cdot\partial\mu
\right)
-
\sigma
=
0.
\]
Since $\sum_\tau |d_{\tau\sigma}|=p+1$, $\sum_\eta |s_{\eta\tau}|\leq q$, and $\sum_\mu |r_{\mu\sigma}|\leq q$, the first sum contains, with multiplicity, at most $(p+1)q$ signed $p$-simplices.
After expanding each boundary $\partial\mu$, the second sum contains at most $(p+2)q$ signed $p$-simplices.
Including the final copy of $\sigma$, the total number of signed $p$-simplices in the relation is therefore at most $(2p+3)q+1 \leq (2n+3)q+1$.
Set $\nu_n(q):=(2n+3)q+1$.
The manifold $V_\sigma$ is constructed by taking the product cobordism on these signed $p$-simplices and attaching one $1$-handle $\Delta^p\times I$ for each algebraically canceling pair.
Thus there are at most $\nu_n(q)$ product pieces and at most $\lfloor\nu_n(q)/2\rfloor$ cancellation handles.
Since the standard cell structure of $\Delta^p\times I$ has at most $3\cdot u_n$ cells, we may take
\[
b_n(q)
:=
3
\cdot
u_n
\cdot
\left(
\nu_n(q)
+
\left\lfloor\frac{\nu_n(q)}{2}\right\rfloor
\right)
\]
as an upper bound for the number of cells of each $V_\sigma$.
In particular,
$
b_n(q)
=
\exp(O(n))
\cdot
(1+q)$.

Cha and the author then form
$
V
=
(-1)^{n-p}
\bigcup_\sigma
Y_\sigma\times V_\sigma$.
Using the product cell structures, the number of cells in $V$ is therefore at most
$
u_n
\cdot
s_{n,p}
\cdot
b_n(q)
\cdot
N_p$.

\subsubsection*{The block $U$.}
The remaining terms in the chain-homotopy identity are realized geometrically by
\[
U
=
\left(
\bigcup_{\eta,\tau}
s_{\eta\tau}\cdot Z_\tau\times\Delta_\eta^p
\right)
\cup
\left(
\bigcup_{\mu,\sigma}
(-1)^{n-p}
r_{\mu\sigma}\cdot Y_\sigma\times\Delta_\mu^{p+1}
\right).
\]
The first family consists of products $Z_\tau\times\Delta_\eta^p$ and occurs with multiplicity controlled by
$
\sum_\eta |s_{\eta\tau}|
\leq
q$.

Since the modified union of the $Z_\tau$ has at most $4\cdot u_n\cdot s_{n,p}\cdot N_p$ cells and $\Delta^p$ has at most $u_n$ cells, these pieces contribute at most
$
4
\cdot
u_n^2
\cdot
q
\cdot
s_{n,p}
\cdot
N_p$ cells.

The second family consists of products $Y_\sigma\times\Delta_\mu^{p+1}$ and occurs with multiplicity controlled by
$
\sum_\mu |r_{\mu\sigma}|
\leq
q$.
Since the compatible cell structure on $\bigcup_\sigma Y_\sigma$ has at most $u_n\cdot s_{n,p}\cdot N_p$ cells and $\Delta^{p+1}$ has at most $u_{n+1}$ cells, these pieces contribute at most
$
u_n
\cdot
u_{n+1}
\cdot
q
\cdot
s_{n,p}
\cdot
N_p$ cells.

The product cell structures on $V$ and $U$ agree on $\partial_+V = \partial_-U$, and hence define a cell structure on $Z = V\cup_{\partial_+V=\partial_-U}U$.

\subsubsection*{The skeleton-lowering bordism $W_p$.}
The incoming boundary of $Z$ is
$
\partial_-Z
=
\bigcup_\sigma
Y_\sigma\times\Delta_\sigma^p$.
Using the product regular neighborhoods $Y_\sigma\times\sigma_0\subset M_p$, Cha and the author attach $Z$ to the product bordism $M_p\times I$ and define
$
W_p
=
(M_p\times I)
\cup_{\partial_-Z}
Z$.
Its outgoing boundary is
$
M_{p-1}
=
E
\cup_{\partial}
\partial_+Z$,
where
$
E
=
M_p
\setminus
\bigcup_\sigma
\operatorname{int}
\bigl(
Y_\sigma\times\sigma_0
\bigr)$.
The construction is arranged so that $M_{p-1}$ maps to the $(p-1)$-skeleton of the target.

The quantitative argument of~\cite{Cha-Lim:2024-1} does not triangulate $M_p\times I$, $V$, and $U$ independently.
Instead, compatible cell structures are first chosen on these three pieces, and these cell structures agree on all attaching subspaces.
Only after they have been assembled into $W_p$ is a single pulling triangulation applied to the resulting common polytopal cell complex, in the sense used in~\cite[Lemma~2.6]{Cha-Lim:2024-1}.
This compatibility is important for the iteration below.

The cell structure on $M_p$ used above has at most $u_n\cdot s_{n,p}\cdot N_p$ cells, so the product cell structure on $M_p\times I$ has at most
$
3
\cdot
u_n
\cdot
s_{n,p}
\cdot
N_p$ cells.
Combining the preceding estimates, the common cell structure on $W_p$ therefore has at most
$
E_{n,p}(q)
\cdot
N_p$ cells, where
\[
E_{n,p}(q)
:=
s_{n,p}
\Bigl( 3u_n
+
u_n\cdot b_n(q)
+
4u_n^2\cdot q
+
u_n \cdot u_{n+1}\cdot q
\Bigr).
\]
Since $\log s_{n,p}=O(n^2)$ and $b_n(q)=\exp(O(n))\cdot(1+q)$, we have, uniformly for $1\leq p\leq n$,
\[
\log E_{n,p}(q)
=
O(n^2)
+
O\!\bigl(\log(1+q)\bigr).
\]

\subsection{The top-dimensional complexity of one skeleton-lowering step}

It remains to pass from the common cell structure on $W_p$ to the pulling triangulation used in~\cite{Cha-Lim:2024-1}.
We retain here the vertex estimates used in that construction rather than optimizing them.

The cells $Y_\sigma\cap A$ arising in simplicial-cellular transversality have at most $(n+1)^{n+1}$ vertices, and the pullback cells appearing in the subdivision used to construct the $Z_\tau$ have at most $2^{2n+2}$ vertices.
The remaining cells are obtained from these cells and from standard simplices by taking products with simplices of dimension at most $n+1$ or with an interval.
Set
\[
\alpha_n
:=
\max
\left\{
(n+1)^{n+1},
2^{2n+2},
2(n+2)
\right\}
\]
and
$
\beta_n
:=
\bigl(
2(n+2)\cdot \alpha_n
\bigr)^2$.
The product-cell estimates used in~\cite{Cha-Lim:2024-1} show that every cell of the common cell structure on $W_p$ has at most $\beta_n$ vertices.
Moreover,
$
\log \beta_n
=
O\!\bigl(n\log(n+2)\bigr)$.

The pulling triangulation introduces no new vertices.
Therefore every $(n+1)$-cell containing at most $\beta_n$ vertices contributes at most
$
\Theta_n
:=
\binom{\beta_n}{n+2}
$
simplices of dimension $n+1$.
It follows that
$
f_{n+1}(W_p)
\leq
\Lambda_{n,p}(q)
\cdot
N_p$,
where
$
\Lambda_{n,p}(q)
:=
\Theta_n
\cdot
E_{n,p}(q)$.
Since
\[
\log \Theta_n
\leq
(n+2)\log \beta_n
=
O\!\bigl(n^2\log(n+2)\bigr),
\]
we obtain
\[
\log \Lambda_{n,p}(q)
=
O\!\bigl(n^2\log(n+2)\bigr)
+
O\!\bigl(\log(1+q)\bigr).
\]

We now use the uniform chain-homotopy bound $Q_n$ obtained above and set
\[
\Lambda_{n,p}
:=
\Lambda_{n,p}(Q_n),
\qquad
\Lambda_n
:=
\max_{1\leq p\leq n}\Lambda_{n,p}.
\]
Since
$
\log Q_n
=
O\!\bigl(n\log(n+2)\bigr)$, the preceding estimate gives
$
\log \Lambda_n
=
O\!\bigl(n^2\log(n+2)\bigr)$.
Thus every application of~\cite[Proposition~4.2]{Cha-Lim:2024-1} may be carried out so that
\[
f_{n+1}(W_p)
\leq
\Lambda_n
\cdot
f_n(M_p).
\]

\subsection{Iteration over the skeleta}

The preceding one-step estimate is not yet the quantitative estimate for the final Cha--Lim bordism.
The construction starts with $M_n=M$ and applies the skeleton-lowering procedure successively,
\[
M_n
\rightsquigarrow
M_{n-1}
\rightsquigarrow
\cdots
\rightsquigarrow
M_1
\rightsquigarrow
M_0,
\]
where
\[
W_p\colon M_p\rightsquigarrow M_{p-1},
\qquad
p=n,n-1,\ldots,1.
\]
The triangulation produced on the outgoing boundary $M_{p-1}$ of $W_p$ is the triangulation used as the input to the next skeleton-lowering step.
Thus the complexity increase at one stage is inherited by all subsequent stages.

Choose the initial triangulation of $M=M_n$ so that
$
N_n
:=
f_n(M_n)
=
\Delta(M)$,
and put
$
N_p
:=
f_n(M_p)$
for the triangulations occurring in the induction.
Since $M_{p-1}$ is a boundary component of the triangulated $(n+1)$-manifold $W_p$, every $n$-simplex of $M_{p-1}$ is a facet of an $(n+1)$-simplex of $W_p$.
An $(n+1)$-simplex has $n+2$ facets, and therefore
\[
N_{p-1}
\leq
(n+2)
\cdot
f_{n+1}(W_p)
\leq
(n+2)
\cdot
\Lambda_{n,p}
\cdot
N_p.
\]
The simplicial-cellular approximation used between successive applications of~\cite[Proposition~4.2]{Cha-Lim:2024-1} changes the structure map by homotopy on the same triangulated manifold, so no additional subdivision factor is introduced at this point.

Iterating the preceding inequality gives
\[
N_p
\leq
\left(
\prod_{j=p+1}^{n}
(n+2)
\cdot
\Lambda_{n,j}
\right)
\cdot
\Delta(M).
\]
Consequently,
\[
f_{n+1}(W_p)
\leq
\Lambda_{n,p}
\cdot
\left(
\prod_{j=p+1}^{n}
(n+2)
\cdot
\Lambda_{n,j}
\right)
\cdot
\Delta(M).
\]

The final bordism is obtained by concatenation,
\[
W
=
W_n
\cup_{M_{n-1}}
W_{n-1}
\cup_{M_{n-2}}
\cdots
\cup_{M_1}
W_1.
\]
The triangulations on the intermediate boundary manifolds already agree, so this final gluing requires no additional subdivision.
No top-dimensional simplices are identified in the gluing, and hence
$
f_{n+1}(W)
=
\sum_{p=1}^{n}
f_{n+1}(W_p)$.
It follows that
$
f_{n+1}(W)
\leq
C_{\mathrm{acy}}(n)
\cdot
\Delta(M)$,
where one may take
\[
C_{\mathrm{acy}}(n)
:=
\sum_{p=1}^{n}
\Lambda_{n,p}
\cdot
\prod_{j=p+1}^{n}
\bigl(
(n+2)\cdot \Lambda_{n,j}
\bigr).
\]
In particular, since $\Lambda_{n,p}\leq \Lambda_n$,
\[
C_{\mathrm{acy}}(n)
\leq
\Lambda_n
\sum_{r=0}^{n-1}
\bigl(
(n+2)\cdot \Lambda_n
\bigr)^r
\leq
n
\cdot
(n+2)^{n-1}
\cdot
\Lambda_n^n.
\]
Since
$
\log \Lambda_n
=
O\!\bigl(n^2\log(n+2)\bigr)$,
we obtain that $C_{\mathrm{acy}}(n) \leq \exp\!\bigl(O(n^3\log(n+2)) \bigr)$.
Thus the construction of~\cite{Cha-Lim:2024-1}, combined with the improved quantitative Baumslag--Dyer--Heller homotopy of~\cite{Lim:2022-1}, gives a bordism satisfying $\Delta(W) \leq \exp\!\bigl(O(n^3\log(n+2)) \bigr) \cdot \Delta(M)$.

\subsection{Comparison with relative hyperbolization and direct asphericalization}

We now compare the preceding bound with the two geometric constructions available for faithful representations.
For the relative-hyperbolization construction of~\cite{Lim-Weinberger:2023-1}, let
\[
z(m)
:=
3^{m-1}
\cdot
m!
\cdot
(m-1)!^2
\cdot
(m-2)!^2
\cdots
(3!)^2
\cdot
2!.
\]
The suspension of an $n$-manifold with $N$ top-dimensional simplices has $2N$ top-dimensional $(n+1)$-simplices.
The barycentric subdivision contributes a factor $(n+2)!$, and each resulting $(n+1)$-simplex is replaced by a hyperbolized $(n+1)$-simplex containing $z(n+1)$ top-dimensional simplices.
Since the relative-hyperbolization bordism is obtained as a subcomplex of the resulting hyperbolization, for the purpose of this comparison we take the explicit construction constant
$
C_{\mathrm{hyp}}(n)
:=
2
\cdot
(n+2)!
\cdot
z(n+1).
$
Since
\[
\log z(n+1)
=
n\log 3
+
\log((n+1)!)
+
2
\sum_{j=3}^{n}
\log(j!)
+
O(1),
\]
Stirling's formula gives $\log C_{\mathrm{hyp}}(n) = n^2\log n + O(n^2)$.

For the direct toroidal asphericalization constructed in the present paper, Lemma~\ref{lem:geometric} gives
\[
C_{\mathrm{asph}}(n)
:=
(n+3)
\cdot
4^n
\cdot
(n+1)!
\cdot
(n+2)!.
\]
Therefore $\log C_{\mathrm{asph}}(n) = 2n\log n + O(n)$.
In particular,
\[
\frac{
\log C_{\mathrm{asph}}(n)
}{
\log C_{\mathrm{hyp}}(n)
}
\longrightarrow
0
\qquad
\text{as }n\longrightarrow\infty.
\]

The three constructions should be compared together with their respective hypotheses.
The acyclic-group construction of~\cite{Cha-Lim:2024-1} applies to arbitrary representations and the preceding analysis gives the admissible estimate
\[
C_{\mathrm{acy}}(n)
\leq
\exp\!\bigl(
O(n^3\log(n+2))
\bigr).
\]
The relative-hyperbolization construction of~\cite{Lim-Weinberger:2023-1} applies to faithful representations and satisfies
\[
\log C_{\mathrm{hyp}}(n)
=
n^2\log n
+
O(n^2).
\]
The direct asphericalization of the present paper also applies to faithful representations, but satisfies
\[
\log C_{\mathrm{asph}}(n)
=
2n\log n
+
O(n).
\]
Thus, in the common setting of faithful representations, direct asphericalization replaces the superfactorial dimension dependence of relative hyperbolization by factorial-type growth.
The acyclic-group method has the greater generality of allowing arbitrary representations, while its skeleton-by-skeleton realization propagates the quantitative cost of each stage through all subsequent stages.

\bibliographystyle{amsalpha}
\bibliography{research}{}

@preamble{"\def\cprime{$'$} "}

@incollection {Scott-Wall:1979,
    AUTHOR = {Scott, Peter and Wall, Terry},
     TITLE = {Topological methods in group theory},
 BOOKTITLE = {Homological group theory ({P}roc. {S}ympos., {D}urham, 1977)},
    SERIES = {London Math. Soc. Lecture Note Ser.},
    VOLUME = {36},
     PAGES = {137--203},
 PUBLISHER = {Cambridge Univ. Press, Cambridge-New York},
      YEAR = {1979},
      ISBN = {0-521-22729-1},
   MRCLASS = {57M05 (20E06)},
  MRNUMBER = {564422},
MRREVIEWER = {William\ H.\ Jaco},
}

@article {Williams:1963,
    AUTHOR = {Williams, R. F.},
     TITLE = {A useful functor and three famous examples in topology},
   JOURNAL = {Trans. Amer. Math. Soc.},
  FJOURNAL = {Transactions of the American Mathematical Society},
    VOLUME = {106},
      YEAR = {1963},
     PAGES = {319--329},
      ISSN = {0002-9947},
   MRCLASS = {55.25},
  MRNUMBER = {146832},
MRREVIEWER = {H. Suzuki},
       DOI = {10.2307/1993773},
       URL = {https://doi-org.proxy.library.ucsb.edu:9443/10.2307/1993773},
}

@article {Ontaneda:2020,
    AUTHOR = {Ontaneda, Pedro},
     TITLE = {Riemannian hyperbolization},
   JOURNAL = {Publ. Math. Inst. Hautes \'Etudes Sci.},
  FJOURNAL = {Publications Math\'ematiques. Institut de Hautes \'Etudes
              Scientifiques},
    VOLUME = {131},
      YEAR = {2020},
     PAGES = {1--72},
      ISSN = {0073-8301,1618-1913},
   MRCLASS = {53C21 (57R20)},
  MRNUMBER = {4106793},
MRREVIEWER = {Thilo\ Kuessner},
       DOI = {10.1007/s10240-020-00113-1},
       URL = {https://doi.org/10.1007/s10240-020-00113-1},
}

@article {Davis-Januszkiewicz-Weinberger:2001,
    AUTHOR = {Davis, Michael W. and Januszkiewicz, Tadeusz and Weinberger,
              Shmuel},
     TITLE = {Relative hyperbolization and aspherical bordisms: an addendum
              to ``{H}yperbolization of polyhedra'' [{J}.\ {D}ifferential
              {G}eom.\ {\bf 34} (1991), no.\ 2, 347--388; {MR}1131435
              (92h:57036)] by {D}avis and {J}anuszkiewicz},
   JOURNAL = {J. Differential Geom.},
  FJOURNAL = {Journal of Differential Geometry},
    VOLUME = {58},
      YEAR = {2001},
    NUMBER = {3},
     PAGES = {535--541},
      ISSN = {0022-040X,1945-743X},
   MRCLASS = {57Q05 (53C20)},
  MRNUMBER = {1906785},
       URL = {http://projecteuclid.org/euclid.jdg/1090348358},
}

@article {Charney-Davis:1995, 
    AUTHOR = {Charney, Ruth M. and Davis, Michael W.},
     TITLE = {Strict hyperbolization},
   JOURNAL = {Topology},
  FJOURNAL = {Topology. An International Journal of Mathematics},
    VOLUME = {34},
      YEAR = {1995},
    NUMBER = {2},
     PAGES = {329--350},
      ISSN = {0040-9383},
   MRCLASS = {57Q05 (53C23)},
  MRNUMBER = {1318879},
MRREVIEWER = {Colin\ C.\ Adams},
       DOI = {10.1016/0040-9383(94)00027-I},
       URL = {https://doi.org/10.1016/0040-9383(94)00027-I},
}

@incollection {Gromov:1987-1,
    AUTHOR = {Gromov, M.},
     TITLE = {Hyperbolic groups},
 BOOKTITLE = {Essays in group theory},
    SERIES = {Math. Sci. Res. Inst. Publ.},
    VOLUME = {8},
     PAGES = {75--263},
 PUBLISHER = {Springer, New York},
      YEAR = {1987},
      ISBN = {0-387-96618-8},
   MRCLASS = {20F32 (20F06 20F10 22E40 53C20 57R75 58F17)},
  MRNUMBER = {919829},
MRREVIEWER = {Christopher\ W.\ Stark},
       DOI = {10.1007/978-1-4613-9586-7\_3},
       URL = {https://doi.org/10.1007/978-1-4613-9586-7_3},
}

@article{Cha:2014-1,
	author = {Cha, Jae Choon},
	doi = {10.1002/cpa.21597},
	fjournal = {Communications on Pure and Applied Mathematics},
	issn = {0010-3640},
	journal = {Comm. Pure Appl. Math.},
	mrclass = {57R57 (58J22)},
	mrnumber = {3493628},
	mrreviewer = {Michael S. Farber},
	number = {6},
	pages = {1154--1209},
	title = {A topological approach to {C}heeger-{G}romov universal bounds for von {N}eumann {$\rho$}-invariants},
	url = {http://dx.doi.org/10.1002/cpa.21597},
	volume = {69},
	year = {2016}}

@article{Chang-Weinberger:2003-1,
	author = {Chang, Stanley and Weinberger, Shmuel},
	fjournal = {Geometry and Topology},
	issn = {1465-3060},
	journal = {Geom. Topol.},
	mrclass = {57R67 (46L80 58J20 58J28)},
	mrnumber = {1988288 (2004c:57052)},
	mrreviewer = {Thomas Schick},
	pages = {311--319 (electronic)},
	title = {On invariants of {H}irzebruch and {C}heeger-{G}romov},
	volume = {7},
	year = {2003}}

@article{Cheeger-Gromov:1985-1,
	author = {Cheeger, Jeff and Gromov, Mikhael},
	coden = {JDGEAS},
	fjournal = {Journal of Differential Geometry},
	issn = {0022-040X},
	journal = {J. Differential Geom.},
	mrclass = {58G12 (53C20)},
	mrnumber = {MR806699 (87d:58136)},
	mrreviewer = {J{\'o}zef Dodziuk},
	number = {1},
	pages = {1--34},
	title = {Bounds on the von {N}eumann dimension of {$L\sp 2$}-cohomology and the {G}auss-{B}onnet theorem for open manifolds},
	volume = {21},
	year = {1985}}

@book {Cheeger-Gromov:1985-2,
    AUTHOR = {Cheeger, Jeff and Gromov, Mikhael},
    TITLE = {On the characteristic numbers of complete manifolds of bounded curvature and finite volume},
    SERIES = {H.E. Rauch Memorial Volume: Differential Geometry and Complex Analysis},
    VOLUME = {1},
    PUBLISHER = {I. Chavel and H. M. Farkas, Eds., Springer, Berlin-New York},
      YEAR = {1985},
     PAGES = {115--154},
      ISBN = {},
   MRCLASS = {)},
  MRNUMBER = {780040},
MRREVIEWER = {}}

@article{Cochran-Orr-Teichner:1999-1,
	author = {Cochran, Tim D. and Orr, Kent E. and Teichner, Peter},
	coden = {ANMAAH},
	fjournal = {Annals of Mathematics. Second Series},
	issn = {0003-486X},
	journal = {Ann. of Math. (2)},
	mrclass = {57M25 (57M27 57Rxx 58Jxx)},
	mrnumber = {1973052},
	number = {2},
	pages = {433--519},
	title = {Knot concordance, {W}hitney towers and {$L\sp 2$}-signatures},
	volume = {157},
	year = {2003}}

@book{Hatcher:2002-1,
	address = {Cambridge},
	author = {Hatcher, Allen},
	isbn = {0-521-79160-X; 0-521-79540-0},
	mrclass = {55-01 (55-00)},
	mrnumber = {1867354 (2002k:55001)},
	mrreviewer = {Donald W. Kahn},
	pages = {xii+544},
	publisher = {Cambridge University Press},
	title = {Algebraic topology},
	year = {2002}}

@article{Kan-Thurston:1976-1,
	author = { Kan, Daniel M. and Thurston, William P.},
	fjournal = {Topology. An International Journal of Mathematics},
	issn = {0040-9383},
	journal = {Topology},
	mrclass = {55D20},
	mrnumber = {0413089 (54 \#1210)},
	mrreviewer = {J. P. May},
	number = {3},
	pages = {253--258},
	title = {Every connected space has the homology of a {$K(\pi ,1)$}},
	volume = {15},
	year = {1976}}

@book{Lueck:2002-1,
	address = {Berlin},
	author = {L{\"u}ck, Wolfgang},
	isbn = {3-540-43566-2},
	mrclass = {58J22 (19K56 46L80 57Q10 57R20 58J52)},
	mrnumber = {MR1926649 (2003m:58033)},
	mrreviewer = {Thomas Schick},
	pages = {xvi+595},
	publisher = {Springer-Verlag},
	series = {Ergebnisse der Mathematik und ihrer Grenzgebiete. 3. Folge. A Series of Modern Surveys in Mathematics [Results in Mathematics and Related Areas. 3rd Series. A Series of Modern Surveys in Mathematics]},
	title = {{$L\sp 2$}-invariants: theory and applications to geometry and {$K$}-theory},
	volume = {44},
	year = {2002}}

@incollection {Gromov:1999-1,
    AUTHOR = {Gromov, Mikhael},
     TITLE = {Quantitative homotopy theory},
 BOOKTITLE = {Prospects in mathematics ({P}rinceton, {NJ}, 1996)},
     PAGES = {45--49},
 PUBLISHER = {Amer. Math. Soc., Providence, RI},
      YEAR = {1999},
   MRCLASS = {57N65 (53C23 55P99)},
  MRNUMBER = {1660471},
MRREVIEWER = {Vagn Lundsgaard Hansen}}

@article {Chambers-Dotterrer-Manin-Weinberger:2018-1,
    AUTHOR = {Chambers, Gregory R. and Dotterrer, Dominic and Manin, Fedor
              and Weinberger, Shmuel},
     TITLE = {Quantitative null-cobordism},
      NOTE = {With an appendix by Manin and Weinberger},
   JOURNAL = {J. Amer. Math. Soc.},
  FJOURNAL = {Journal of the American Mathematical Society},
    VOLUME = {31},
      YEAR = {2018},
    NUMBER = {4},
     PAGES = {1165--1203},
      ISSN = {0894-0347},
   MRCLASS = {57Q20 (53C23 57R75)},
  MRNUMBER = {3836564},
MRREVIEWER = {Greg Friedman},
       DOI = {10.1090/jams/903},
       URL = {https://doi-org.proxy.library.ucsb.edu:9443/10.1090/jams/903}}

@article {Lim:2022-1,
    AUTHOR = {Lim, Geunho},
     TITLE = {Enhanced bounds for rho-invariants for both general and spherical 3-manifolds},
   JOURNAL = {J. Topol. Anal.},
  FJOURNAL = {Journal of Topology and Analysis},
    VOLUME = {16},
      YEAR = {2024},
    NUMBER = {3},
     PAGES = {409-459},
      ISSN = {1793-5253,1793-7167},
   MRCLASS = {},
  MRNUMBER = {},
       DOI = {10.1142/S1793525322500029},
       URL = {https:///dx.doi.org/10.1142/S1793525322500029},
}

@article {Lim-Weinberger:2023-1,
    AUTHOR = {Lim, Geunho and Weinberger, Shmuel},
     TITLE = {Bounds on {C}heeger--{G}romov invariants and simplicial
              complexity of triangulated manifolds},
   JOURNAL = {J. Reine Angew. Math.},
  FJOURNAL = {Journal f\"{u}r die Reine und Angewandte Mathematik. [Crelle's
              Journal]},
    VOLUME = {808},
      YEAR = {2024},
     PAGES = {271--297},
      ISSN = {0075-4102,1435-5345},
   MRCLASS = {99-06},
  MRNUMBER = {4708122},
       DOI = {10.1515/crelle-2024-0003},
       URL = {https://doi.org/10.1515/crelle-2024-0003},
}

@article {Boissonnat-Dyer-Ghosh:2018-1,
    AUTHOR = {Boissonnat, Jean-Daniel and Dyer, Ramsay and Ghosh, Arijit},
     TITLE = {Delaunay triangulation of manifolds},
   JOURNAL = {Found. Comput. Math.},
  FJOURNAL = {Foundations of Computational Mathematics. The Journal of the
              Society for the Foundations of Computational Mathematics},
    VOLUME = {18},
      YEAR = {2018},
    NUMBER = {2},
     PAGES = {399--431},
      ISSN = {1615-3375,1615-3383},
   MRCLASS = {57R05 (52B70 54B15)},
  MRNUMBER = {3777784},
MRREVIEWER = {Vladimir\ Aleksandrovich\ Klyachin},
       DOI = {10.1007/s10208-017-9344-1},
       URL = {https://doi.org/10.1007/s10208-017-9344-1},
}

@misc{Manin-Weinberger:2023-1,
	author = {Manin, Fedor and Weinberger, Shmuel},
	howpublished = {arXiv:2311.16389},
	title = {Quantitative {PL} bordism},
	year = {2023},
}

@article{Cha-Lim:2024-1,
  author  = {Cha, Jae Choon and Lim, Geunho},
  title   = {Quantitative bordism over acyclic groups and {Cheeger--Gromov} $\rho$-invariants},
  journal = {Geom. Topol.},
  volume  = {30},
  number  = {5},
  year    = {2026},
  pages   = {1899--1929},
  doi     = {10.2140/gt.2026.30.1899}
}

@article {Davis-Januszkiewicz:1991-1,
    AUTHOR = {Davis, Michael W. and Januszkiewicz, Tadeusz},
     TITLE = {Hyperbolization of polyhedra},
   JOURNAL = {J. Differential Geom.},
  FJOURNAL = {Journal of Differential Geometry},
    VOLUME = {34},
      YEAR = {1991},
    NUMBER = {2},
     PAGES = {347--388},
      ISSN = {0022-040X},
   MRCLASS = {57Q05 (53C20)},
  MRNUMBER = {1131435},
MRREVIEWER = {Werner Ballmann},
       URL = {http://projecteuclid.org.proxy.library.ucsb.edu:2048/euclid.jdg/1214447212},
}

@article {Hausmann:1981,
    AUTHOR = {Hausmann, Jean-Claude},
     TITLE = {On the homotopy of nonnilpotent spaces.},
   JOURNAL = {Math. Z.},
  FJOURNAL = {Mathematische Zeitschrift},
    VOLUME = {178},
      YEAR = {1981},
    NUMBER = {1},
     PAGES = {115--123},
      ISSN = {},
   MRCLASS = {},
  MRNUMBER = {},
MRREVIEWER = {},
       DOI = {},
       URL = {},
}

@article {Maunder:1981,
    AUTHOR = {Maunder, C. R. F.},
     TITLE = {A short proof of a theorem of {K}an and {T}hurston},
   JOURNAL = {Bull. London Math. Soc.},
  FJOURNAL = {The Bulletin of the London Mathematical Society},
    VOLUME = {13},
      YEAR = {1981},
    NUMBER = {4},
     PAGES = {325--327},
      ISSN = {0024-6093,1469-2120},
   MRCLASS = {55P20},
  MRNUMBER = {620046},
MRREVIEWER = {J.\ P.\ May},
       DOI = {10.1112/blms/13.4.325},
       URL = {https://doi.org/10.1112/blms/13.4.325},
}
\def\MR#1{}

\end{document}